\documentclass[%
 reprint,
 amsmath,amssymb,
 aps,
floatfix,
]{revtex4-1}

\usepackage{graphicx}
\usepackage{dcolumn}
\usepackage{bm}
\usepackage[running]{lineno}
\usepackage{xcolor}
\usepackage{amsmath, amsfonts, amssymb, amsthm, mathtools}
\usepackage{graphicx}

\newtheorem{theorem}{Theorem}
\newtheorem{lemma}{Lemma}

\theoremstyle{remark}
\newtheorem{Rem}{Remark}

\usepackage{framed}

\newcommand{\concept}[1]{``{#1}''}
\usepackage{fullpage}
\usepackage{algorithm}

\newcommand{\refSec}[1]{sec.\,\ref{#1}}

\newcommand{\refSupp}[1]{SI\,\refSec{#1}}
\newcommand{\refEqn}[1]{eqn. \eqref{eqn:#1}}
\newcommand{\refFig}[1]{figure\,\ref{#1}}

\newcommand{\refTbl}[1]{table\,\ref{tbl:#1}}

\newcommand{\RefFig}[1]{Figure\,\ref{#1}}

\newcommand{\eqn}[1]{eqn. (\ref{#1})}
\newcommand{\Real}{\mathbb{R}}
\newcommand{\Natural}{\mathbb{N}}

\newcommand{\Complex}{\mathbb{C}}

\newcommand{\Cont}{\mathcal{C}}

\newcommand{\Spc}[1]{\mathbf{#1}}
\newcommand{\Xs}{\Spc{X}}

\newcommand{\Sph}{\mathrm{S}}

\newcommand{\Deriv}{{\mathbf d}}

\newcommand{\defeq}{:=}

\newcommand{\T}{\mathsf{T}}

\DeclarePairedDelimiter\norm{\lVert}{\rVert}

\newcommand{\slot}{\,\cdot\,} 
\newcommand{\R}{\mathbb{R}}
\newcommand{\N}{\mathbb{N}}
\newcommand{\E}{\mathbb{E}}

\newcommand{\inp}[2]{\left\langle{#1},{#2}\right\rangle}
\newcommand{\paren}[1]{\left({#1}\right)}

\newcommand{\dform}[1]{\mathrm{\mathbf{d}}{#1}}
\newcommand{\Der}[1]{\mathrm{d}{#1}}
\newcommand{\D}{\mathsf{D}}
\newcommand{\dist}[2]{d\left({#1},{#2}\right)}

\newcommand{\cycle}{\Gamma}
\newcommand{\dT}{d\phi}
\newcommand{\thetaFrm}{d\theta}

\newcommand{\ToF}{Temporal 1-Form}
\newcommand{\Phaser}{\textit{Phaser}}

\newcommand{\FlwSym}{\Phi}
\newcommand{\Flw}[2]{\FlwSym_{#1}({#2})}

\newcommand{\AP}{\mathrm{P}}
\renewcommand{\mod}{{\mathrm{mod}~}}
\newcommand{\dTe}{\widehat{\dT}}
\newcommand{\bo}{\mathcal{O}}
\newcommand{\F}{\mathcal{F}}
\newcommand{\eJ}{\widehat{J_N}}
\usepackage[T1]{fontenc}

\renewcommand{\refSec}[1]{sec.\,\ref{sec:#1}}

\renewcommand{\refFig}[1]{figure\,\ref{fig:#1}}
\renewcommand{\RefFig}[1]{Figure\,\ref{fig:#1}}
\newif\ifsupresscomments
\supresscommentstrue 
\ifsupresscomments
  \newcommand{\citeWhat}[1]{}
  \newcommand{\SR}[1]{}
  \newcommand{\SW}[1]{}
  \newcommand{\MK}[1]{[}
  \newcommand{\TBD}[1]{}
\else
  \newcommand{\citeWhat}[1]{[{\color{blue}CITE:\color{red}{#1}}]}
  \newcommand{\SR}[1]{[{\color{magenta}(S.R.) \color{blue}\textsc{#1}}]}
  \newcommand{\SW}[1]{[{\color{magenta}(S.W.) \color{red}\textsc{#1}}]}
  \newcommand{\MK}[1]{[{\color{magenta}(M.K.) \color{cyan}\textsc{#1}}]}
  \newcommand{\TBD}[1]{[{\color{red}TBD:~{#1}}]}
\fi

\renewcommand{\inp}[2]{#1 \cdot #2}

\begin{document}

\title{Estimating Phase from Observed Trajectories Using the \ToF}
\author{Simon Wilshin} 
\altaffiliation{Royal Vet College, London, UK}
\author{Matthew D. Kvalheim}
\altaffiliation{University of Pennsylvania, Philadelphia, PA, USA}
\author{Clayton Scott}
\altaffiliation{University of Michigan, Ann Arbor, MI, USA}
\author{Shai Revzen}
\altaffiliation{University of Michigan, Ann Arbor, MI, USA}

\date{\today}

\begin{abstract}
Oscillators are ubiquitous in nature, and usually associated with the existence of an \concept{asymptotic phase} that governs the long-term dynamics of the oscillator. %
We show that asymptotic phase can be estimated using a carefully chosen series expansion which directly computes the phase response curve and provide an algorithm for estimating the co-efficients of this series.
Unlike all previously available data driven phase estimation methods, our algorithm can: (i) use observations that are much shorter than a cycle;  (ii) recover phase within any forward invariant region for which sufficient data are available; (iii) recover the phase response curves (PRC-s) that govern weak oscillator coupling; (iv) show isochron curvature, and recover nonlinear features of isochron geometry.
Our method may find application wherever models of oscillator dynamics need to be constructed from measured or simulated time-series.
\end{abstract}

\maketitle

\section{Phase as an organizing principle}
\label{sec:phaseasorganising}
Our interest in oscillators arose from their utility as models of animal locomotion \citep{revzen2015data, seipel2017conceptual}, but oscillators appear in virtually every physical science: in biological models \citep{kvalheim2021families} at all scales, from the coupled neuronal firing \cite{hoppensteadt2012weakly} to coupled oscillations of predator and prey populations \cite{may1972limit}; in chemistry \cite{epstein1998introduction}; in physics \cite{glazek2002limit}; in electrical \cite{van1934nonlinear}, civil \cite{duffing1918erzwungene}, and mechanical \cite{stoker1950nonlinear} engineering; etc.
In the typical case that transients decay at an exponential rate, the underlying mathematical structure of oscillators shares many properties across all these cases, in particular the existence of an \concept{asymptotic phase} which encodes the long term outcome of any (recoverable) perturbation.
We have discovered that by representing phase as a \concept{\ToF}, we could bring to bear interpolation tools from machine learning to allow asymptotic phase to be computed in a data driven way---directly from time series of measurements---and without a need to know the governing equations.

This will allow investigators to perform \concept{phase reduction} directly from experimental measurements, and thereby construct models of oscillatory systems and how they couple to each other for a broad swath of systems.
In this paper we present our algorithm applied to simulations with known ground truth, and some classical oscillators (Fitzhugh-Nagumo oscillator \cite{fitzhugh1961impulses,nagumo1962active}, Selkov oscillator \cite{sel1968self}) from the literature.
Additional simulation tests, and examples of use with animal locomotion data and chemical oscillator data are included, with the details of our simulations in the supplemental information (SI). These tests demonstrate the broad applicability of our methods.

\section{Differential equations}
Mathematical models in the physical sciences frequently use differential equations of order 1 or higher.
Newton's laws, and the equations governing electrical circuits are usually written as second order systems $\ddot x = g(x, \dot x)$, whereas first order systems govern $\dot x = f(x)$ simple chemical reactions.
By re-writing the second order equations in terms of positions and velocities (or momenta), the former can be reduced to a special case of the latter, and we will therefore consider first order systems.

Sometimes the solutions of $\dot x = f(x)$ are periodic.
We consider the case that the curve traced by such a solution attracts nearby solutions.
More specifically, we consider only the typical case that this convergence occurs at an exponential rate, which is equivalent to the periodic solution being (normally) \concept{hyperbolic}; for brevity, we will simply refer to such periodic solutions as \concept{limit cycles}.
We use the term \concept{oscillator} to describe the dynamics $f(\cdot)$ restricted to the set of initial conditions which, after a sufficiently long relaxation period, converge arbitrarily close to the designated limit cycle.
That set of initial conditions which evolve toward the limit cycle is termed the \concept{stability basin} of the limit cycle or the domain of the oscillator.

Once the system settles on (or arbitrarily near to) the limit cycle it oscillates with the same pattern repeatedly, allowing points from this closed curve to be invertibly mapped onto the unit circle so that they cycle at a constant angular rate $\omega$.
The angle of the image on this circle is what is commonly refered to as the \concept{phase (angle)} of the oscillator at that state.

This phase can be naturally extended to the entire stability basin thanks to the following observation: all initial conditions converge to the limit cycle, but points on the limit cycle forever remain apart.
To model the long term behavior of an oscillator we could represent each off-cycle point by a point on the cycle which most closely represents its long term behavior, constructing a phase map.
In the typical case that convergence to the limit cycle is at an exponential rate, such long-term representatives actually exist (and are unique), so, by extension, the phase angle of the representative can be taken to be the phase angle of all points it represents.

Since we can assign to every point in our stability basin a phase angle, we can divide up the stability basins into sections (hypersurfaces) with the same phase.
In the mathematics literature these are called \concept{isochrons}. %
The fact that isochrons are manifolds and therefore have no cusps or kinks is a non-trivial result \citep{shGuc75}. %
For a $D$ dimensional system this will be a set of $D-1$ dimensional hypersurfaces on each of which phase is constant.
In Euclidean space the derivatives of the phase are the normals of these surfaces.
These derivatives at the intersection points with the limit cycle govern the long term sensitivity of the system to small perturbations, and are called the \concept{phase response curves}.
They provide a complete first order description of the long-term responses of the oscillator to small perturbations.

\section{Estimating phase by series expansion}
We note that for our deterministic systems our phase, $\phi$ satisfies the following condition: $\phi(x(t)) = \omega t + \phi(x(0))$ when following a trajectory $x(\slot)$ for a system with angular frequency $\omega$.
This is obviously true on the limit cycle.
Off of the limit cycle, $x(0)$ and $x(t)$ are $t$ time units apart on the same solution if and only if $x(s)$ and $x(s+t)$ are as well for any $s$.
Considering $s$ values sufficiently large that $x(s)$ and $x(s+t)$ are effectively on the limit cycle, the result follows.
Rearranging slightly we have:
\begin{equation}
\phi(x(s+t))-\phi(x(s))=\omega t
\end{equation}
Taking the limit $t \to 0$:%
\begin{equation}\label{eqn:nablim}
\omega = \lim_{t\rightarrow 0} \frac{\phi(x(s+t))-\phi(x(s))}{t} = \left.\frac{d}{dt}\phi(x(t))\right|_s
\end{equation}
Since $x(s)$ is an arbitrary point, and $\dot x(s) = f(x(s))$ we obtain from the chain rule applied to \refEqn{nablim}:
\begin{equation}
\label{eqn:definingequation}
\inp{\dT(x)}{ f(x)} = \omega
\end{equation}
Here $d$ is intended to evoke the exterior derivative $\dform{}$\footnote{
  This is the exterior derivative and for Riemannian spaces when acting on a scalar is the operator $\nabla$ with its contravariant index lowered by the metric.
  Since our method is not metric dependent we use $\dform{}$ here, however readers unfamiliar with differential geometry can treat $\dform{}$ as $\nabla$.
}, which represents the set of first order partial derivatives when acting on a scalar.
But $\dT$ is actually not \emph{globally} the exterior derivative of any smooth real-valued function ($\phi$ is not continuous on all of $\Xs$), and we reserve the boldface $\dform{}$ for actual exterior derivatives.

This form has a few advantages.
First, it permits us to write down a point-wise condition our phase function must observe, rather than a condition about its differences along trajectories.
Second, phase angle has an arbitrary gauge in that it is defined only up to choice of a point at phase $0$, but the form $\dT$ does not and is, at least potentially, a physical quantity.
This coordinate free form which we dubbed the \concept{\ToF} (\refSupp{mathbg}) is unique for each oscillator, is mathematically equivalent to other representations of asymptotic phase (Theorem~\ref{th:phase-reps-relate} in \refSupp{equiv}), and can be approximated using our algorithm with roughly the same statistical efficiency as the law of large numbers for data sampling and system and measurement noise (Theorems~\ref{th:cost_funcs_close} and \ref{th:estimated_isochrons_approach_true_isochrons} and Remark~\ref{rem:convergence-rate} in \refSupp{tof_from_uncertain}).
This provable convergence property is unique among all data-driven phase estimation algorithms we are aware of.

As mentioned, the function $\phi$ does not have a conventional derivative everywhere.
First, it suffers a jump discontinuity when moving from $2\pi$ to $0$.
However, a suitable derivative can be calculated at any point by noting that:
(1) Under a different gauge $\tilde \phi(x) := \phi(x) - \theta_0$ for phase angle, the jump would occur at the isochron $\phi(x)=\theta_0$ instead of $\phi(x)=0$.
(2) Wherever they are both defined,  $\dform{\phi}= \dform{\tilde \phi}$.
Thus we may, in a self-consistent way, define $\dT$ on all of $\Xs$ to be given by  $\dform{\phi}$ or $\dform{\tilde{\phi}}$ depending on which quantity exists. 

Second, phase is undefined at the boundaries of the stability basin (and everywhere else outside the stability basin).
If the solution starting at point $y$ does not return to the limit cycle, $\phi$ is not defined at $y$.
If, in addition, $y = \lim_{k\to\infty} x_k$ of points $x_k$ that do return to the limit cycle, then $\phi(x_k)$ need not converge \footnote{%
  While it need not, it could; e.g. $\dot \theta = 1 + c r; \dot r = (r-2)(r-1)r$ has phase identically $\theta$ for $c=0$, but phase diverges at $r=2$ when $c\neq 0$.
}.

\section{Algorithm}\label{sec:algo}

\begin{figure*}[h]
 \centering
 \includegraphics[width=\textwidth]{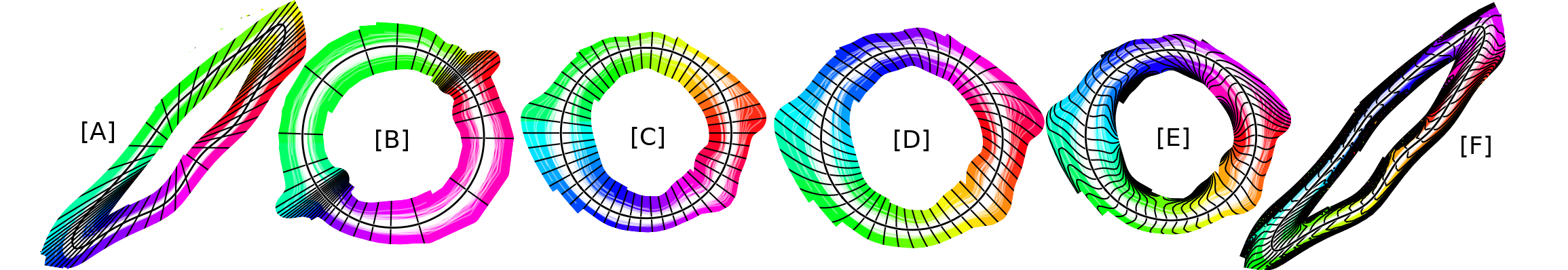}
 \caption{Steps of the algorithm, showing phase (as color; isochrons as contours), trajectories (white traces), and a limit cycle estimate (dark line) as applied to a randomly generated simulation system.%
 We projected the data down to its two principal components, and then built a limit cycle model from the polar angle in the circulation plane [A].
 Using this we rectified the radius to the unit circle [B], and then rectified the phase to a constant circulation rate [C]. %
 We then add basis terms (up to order $2$ [D], order $6$ [E]) and convert the resulting series approximation back to the original coordinates [F].
}
\label{fig:stepbystep}
\end{figure*}

Our algorithm takes as input pairs $(x_k,\dot x_k)$ $k=1\ldots N$ of observed states and state velocities, and learns both an angle valued function that takes states to phase, and a matching vector valued function giving the \ToF.

We perform our algorithm in the following way. First we rectify the data by constructing a smooth change of coordinates that maps an estimated limit cycle to the unit circle at constant angular rate.
Then subtract this constant circulation and expand the remainder with vector valued basis functions in the rectified coordinates and select their coefficients by minimizing the cost of a loss function which is zero for the residual produced from the true \ToF.
The basis functions we use here are derivatives of known scalar valued functions.
The sum of the scalar valued functions and the phase angle contribution of the circulation is our resulting angle valued phase estimate.

\subsection{Rectification}
The first step of our algorithm is fairly similar to Revzen \& Guckenheimer \cite{RevGuk08} in that it estimates the limit cycle, and computes a phase estimate that evolves uniformly on that limit cycle estimate.
However, here we use this estimate to construct a smooth coordinate change that rectifies the data so that the estimated limit cycle is the unit circle in the first two coordinates.

In particular, we translate the data to zero mean, then use principal component analysis (PCA) on the positions $x_k$ to rotate the two major covariance axes into the first two coordinates.
We assume that in these new coordinates the trajectories wind around the origin, and that data is mostly constrained to an annulus in this 2D \concept{circulation plane} comprising the first two coordinates.
Data which does not meet this assumption would require pre-processing before it could be used---some other transformation, more elaborate than PCA, would be needed to bring it to meet the circulation requirement.

Consider the data in cylindrical coordinates: scalar angle $\theta$ and radius $r$ for the circulation plane, and cartesian $z$ with the remaining coordinates of each data point (see \refFig{stepbystep} [A] showing such data).
We fit Fourier series $\hat r(\theta)$ and $\hat z(\theta)$ to the data, and thereby constructed an approximate model of the limit cycle.
This also allowed us to estimate the period $T$ as the rate of circulation of this limit cycle model.
Using the same approach as \textit{Phaser} \cite{RevGuk08}, we constructed a map $\hat\phi(\theta)$ so that $\left|2\pi t/T-(\hat\phi+\theta)\right|^2$ is minimised.
Using this we constructed the map mapping the input $x$ coordinates to rectified coordinates
\begin{align}
  (\theta,r,z) \mapsto (\theta - \hat{\phi}(\theta), r/\hat r(\theta), z - \hat z(\theta))
\end{align}
which rectifies the data to motion on the unit circle (see \refFig{stepbystep} [B]) at close to a constant angular rate (see \refFig{stepbystep} [C]).
We refer to the state $x$ in its rectified coordinates by the symbol $q$.

\subsection{Approximation by topologically motivated basis functions}

We reconstructed the \ToF\ by a series approximation using a basis \footnote{%
 We use the term \concept{basis} informally, as it is used e.g. in the notion of \concept{radial basis functions} in machine learning, to mean a collection of functions whose finite linear spans are dense in the function space of interest. %
} described below.
Most of the forms comprising this basis were themselves derivatives of real-valued functions; we will refer to such as \concept{basis forms}.

Unfortunately, it is not possible to reconstruct the \ToF\ from any linear combination of basis forms, for the following reasons.
Consider the change in the value of $\phi$ after a single cycle on the limit cycle.
For a single cycle this will be $2\pi$ (or $1$, or $360^\circ$ depending on the arbitrary units used).
This difference can be computed using a line integral of the \ToF\, which amounts to just integrating \refEqn{definingequation}.
If we build $\dT$ exclusively from sums of derivatives of any real-valued functions, this integral will necessarily be zero, since the sum of the integrals will just be the integral of the sum.
This demonstrates that the \ToF\ is not exact, i.e. it is not the derivative of any scalar valued function, and therefore cannot be the sum of basis forms.

The key insight that enabled our algorithm comes from algebraic topology  %
\footnote{
  Although the closed 1-form we seek is not exact, algebraic topology teaches us that, on the basin of attraction associated to an oscillator, any two closed 1-forms $\alpha,b$ can be related as $a=\omega b +c$ for some $\omega$ real and $c$ exact.
  Exact forms are exterior derivatives (``gradients'') of scalar functions.
  The authors are deeply indebted to Dan Guralnik for pointing out this insight on the de~Rham cohomology of oscillators.
  That this insight extends from $\Cont^\infty$ closed forms (used in de~Rham cohomology) to $\Cont^0$ closed forms follows from general results on smoothing currents \cite[pp. 61-70]{deRham1984differentiable}; alternatively, it is not difficult to give a direct proof of this result for oscillators.
} %
and is that we can always write our expression for the \ToF\ of one system, $a$, in terms of another \ToF\ for another system, $b$, via $a=\omega b +\dform{c}$ with $\omega$ real and $c$ real-valued. %
We chose to use a trivial oscillator for $b$ by using $d\theta$ with $\theta$ the angle in the plane.
This $d\theta$ is the \ToF\ associated with any oscillator obeying equations of the form $\dot \theta = \omega$ $\dot r = g(r)$.

\begin{table*}
  \begin{align}
  \dT(\theta,r,z) \defeq \thetaFrm +&\sum_\mu m_\mu \dform{v_\mu(\theta,r,z)}
    \label{eqn:series} \\
  &\mu \defeq (i,j,k) ~~ 0\leq i \leq n-2,~ j,k \in \Natural; ~~ v_\mu(\theta,r,z) \defeq  \xi_i(z) \rho_j(r) u_k(\theta) \nonumber
  \end{align}
  \vspace{-3mm}
  \begin{align}
  \text{\textbf{Name}}\hspace{1em} &\text{\textbf{Integral}} & &  \text{\textbf{Differential 1-form}} \nonumber\\
  \label{eqn:dthTerm}\text{angle}\hspace{1em}
  & \theta \approx \text{arctan2}(q_0,q_1)~~\text{\textit{(no global integral exists)}}&
  & \thetaFrm \\
  \label{eqn:FourierTerm}\text{Fourier}\hspace{1em}
  & u_k(\theta) \defeq a_k\cos(k\theta)+b_k\sin(k\theta)&
  &\dform{u_k}(\theta) =  k a_k\sin(k\theta)\thetaFrm-k b_k\cos(k\theta)\thetaFrm\\
  \label{eqn:polyTerm}\text{polynomial}\hspace{1em}
  &\rho_j(r) \defeq (r-1)^j &
  &\dform{\rho_j}(r) = j (r-1)^{j-1} \dform{r} \\
  \label{eqn:zTerm}\text{out of plane}\hspace{1em}
  &\xi_i(z) \defeq z_i ~\text{for i>0, or}~ 1 \text{, for i = 0}&
  &\dform{\xi_i}(z) = \dform{z_i} ~\text{for i>0, or}~ 0 \text{, for i = 0}
  \end{align} 
  \caption{%
    Basis forms and their integrals in rectified coordinates $\theta,r,z$.
    The general $\dT$ structure is in [\ref{eqn:series}]. %
    After the rectification step, $\thetaFrm$ is a trivial expression [\ref{eqn:dthTerm}], and provides the circulation needed to ensure that remaining terms can be exact---the sum of derivatives $\dform{v_\mu}$. %
    The $v_\mu$ basis functions are a first order expansion in $z$, via [\ref{eqn:zTerm}], a Fourier expansion in $\theta$ via [\ref{eqn:FourierTerm}], and a polynomial [\ref{eqn:polyTerm}] expansion in $r$.}
    \label{tbl:series}
\end{table*}

Our approximation algorithm proceeds by optimizing the parameters of \eqn{eqn:series} (in table \ref{tbl:series}) as indicated by the following expression, where the constant $C$ is selected to be consistent with the previously computed period:
\begin{equation}\label{eqn:definingopt}
\arg\min_{\dT}\sum_{i}\paren{\inp{\dT}{\dot{x}_i}-C}^2.
\end{equation}
Here we have multiple observations of state and its derivative, $(x,\dot{x})$ and hence $x$ and $\dot{x}$ have an index, $i$.
First, working in the rectified co-ordinates $q$, we selected a scale for $\thetaFrm(q)$ such that $\int_\cycle \thetaFrm(q)$ is the period of oscillation, where $\cycle\subset \Xs$ is the image of the limit cycle.
\begin{figure}[h]
  \centering
  \includegraphics[width=0.95\columnwidth]{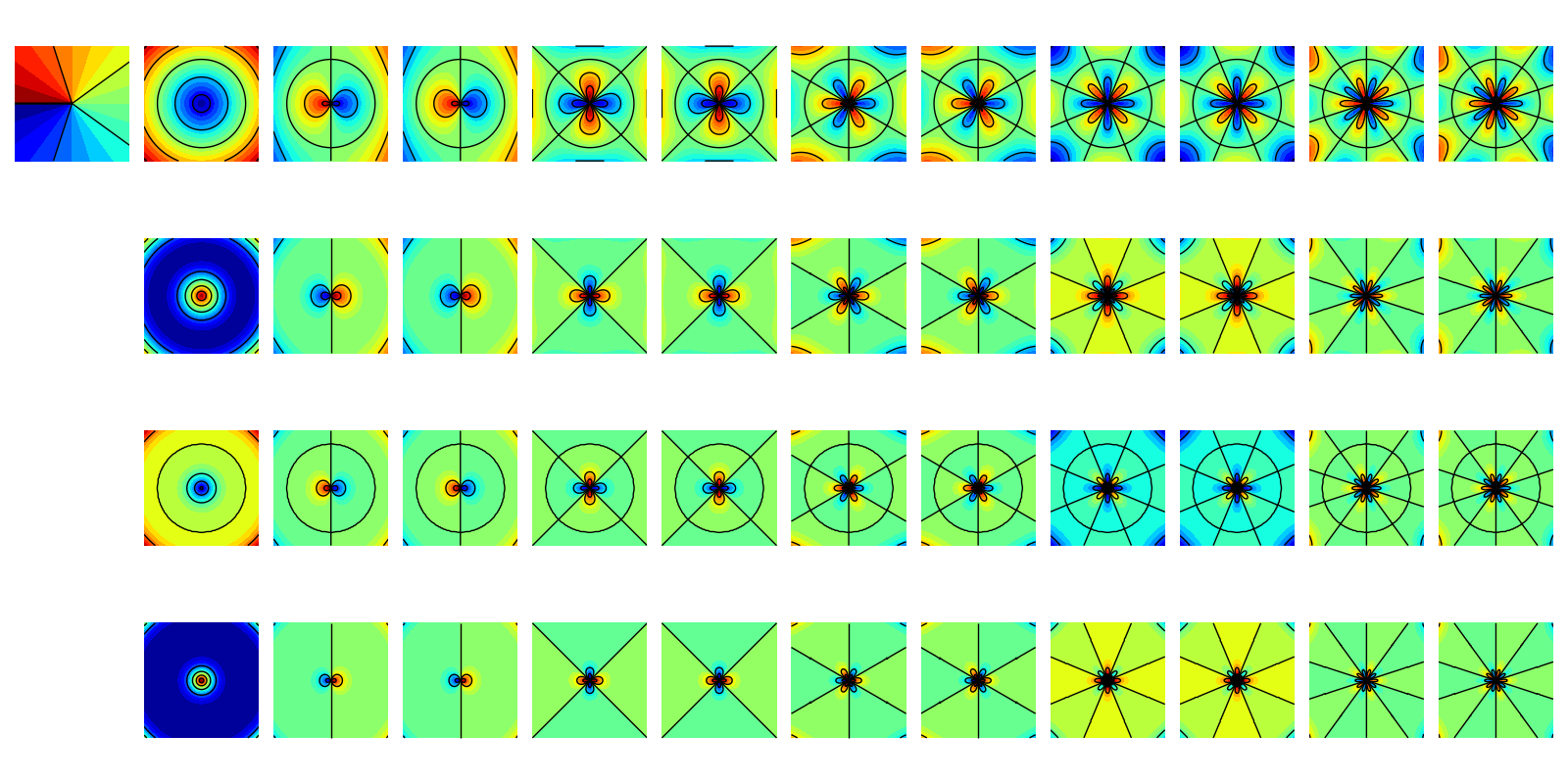}
  \caption{A plot of $u_k(\theta)\rho_j(r)$ terms for orders up to five.
  From left to right the order of the Fourier term increases and the cosine and sine terms alternate.
  From top to bottom we have increasing order of the radial polynomial term.
  In the top right corner is the $d\theta$ term.
  }\label{fig:frPolyVis}
\end{figure}

To find the remaining coefficients $m_\mu$ we solved
\begin{equation}\label{eqn:reg}
\arg\min_{ m_\mu, C } \sum_{ \mu,i } \paren{ m_\mu \inp{ \dform{v_i} (q_i) }{ {\dot q}_i } - \paren{ C - \inp{ \thetaFrm(q_i) }{ {\dot q}_i }}}^2.
\end{equation}
We note that \refEqn{reg} is a conventional ordinary least squares regression problem, and thus $m_\mu$ can be solved for using standard tools.
The contribution of individual Fourier-polynomial terms can be seen in \refFig{frPolyVis}.

\section{Performance of the new algorithm}
\subsection{Performance on 2D, 3D, and 8D simulations}
\label{sec:resultssim}
We used the approach described in Appendix \refSec{groundtruthdata} to produce three simulation systems of dimension two, three and eight.
The first is easy to visualize; the second is the minimal dimension in which general eigenvalues can appear in the Floquet structure; the third is more typical of the dimensionality of biological systems for which we developed this estimation method.
We examined the performance of three different estimates of the phase of these systems and compared the estimated phase to the ground truth provided by $\theta$ of \refEqn{floq}.

The first phase estimator we used is comparable to \concept{event-based} phase estimates commonly used in biological research.
Researchers often use distinguished events, such as voltage levels in the nervous system \cite{danner2015human}, footfall \cite{ting1994dss, jindrich2002dsr}, or anterior extreme position of a limb \cite{wnCru88} to identify the beginning of a cycle and presume that phase evolves uniformly in time between these events.
As our event, we used the zero crossing of the first principal component of the data.
The second phase estimator we used was our previously published \Phaser\ algorithm \cite{RevGuk08}.
The third phase estimator was derived from the \ToF\ as described in this paper, and is referred to as \concept{form phase}.
A visualization of the execution of the form-phase estimation algorithm is depicted in \refFig{stepbystep}.

Even though the phase we wish to reconstruct is that of a deterministic dynamical system, experimental data more closely resembles the sample paths of a stochastic differential equation.
We trained our phase estimation methods on one set of simulated sample paths and tested them on another, generated from the same stochastic dynamical system.
We calculated a model residual by subtracting the ground truth phase from the estimated phases obtained with each method.
We removed the trial-to-trial indeterminate phase offset by taking the circular mean of the model residual to be zero.

The simulations along with the code used to generate them are available at \url{https://purl.archive.org/purl/formphase}.
The results show that the form phase method has lower mean-square residuals than the \Phaser\ method, which in turn has lower mean-square residual than the event-based method; see residuals in \refFig{2dperformance}.

\begin{table}[ht]
\caption{%
  Variance of residual phase for different phase estimation techniques with different initial condition noise and system noise levels. %
}\small
\begin{tabular}
{|c|rrr|rrr|}
\hline
& & noise & & & res. var. & \\
D & initial & system & phase & event & \Phaser & form \\
\hline
2 & 0.1 & 0.01 & 0.1 & 0.128 & 0.0473 & 0.0185 \\
2 & 0.2 & 0.01 & 0.1 & 0.127 & 0.0456 & 0.0205 \\
2 & 0.1 & 0.02 & 0.1 & 0.135 & 0.0480 & 0.0199 \\
2 & 0.1 & 0.01 & 0.2 & 0.222 & 0.0885 & 0.0389 \\
3 & 0.066 & 0.0066 & 0.066 & 0.102 & 0.156 & 0.0164 \\
3 & 0.133 & 0.0066 & 0.066 & 0.102 & 0.0920 & 0.0189 \\
3 & 0.066 & 0.0133 & 0.066 & 0.108 & 0.0863 & 0.0214 \\
3 & 0.066 & 0.0066 & 0.133 & 0.159 & 0.0897 & 0.0252 \\
8 & 0.025 & 0.0025 & 0.025 & 0.0789 & 0.0232 & 0.0340 \\
8 & 0.05 & 0.0025 & 0.025 & 0.0776 & 0.0246 & 0.0455 \\
8 & 0.025 & 0.005 & 0.025 & 0.0816 & 0.0352 & 0.0671 \\
8 & 0.025 & 0.0025 & 0.05 & 0.0889 & 0.0273 & 0.0460 \\
\hline
\end{tabular}\normalsize
\label{tbl:simperformance}
\end{table}

\begin{figure*}[h]
 \centering
 \includegraphics[width=0.9\textwidth]{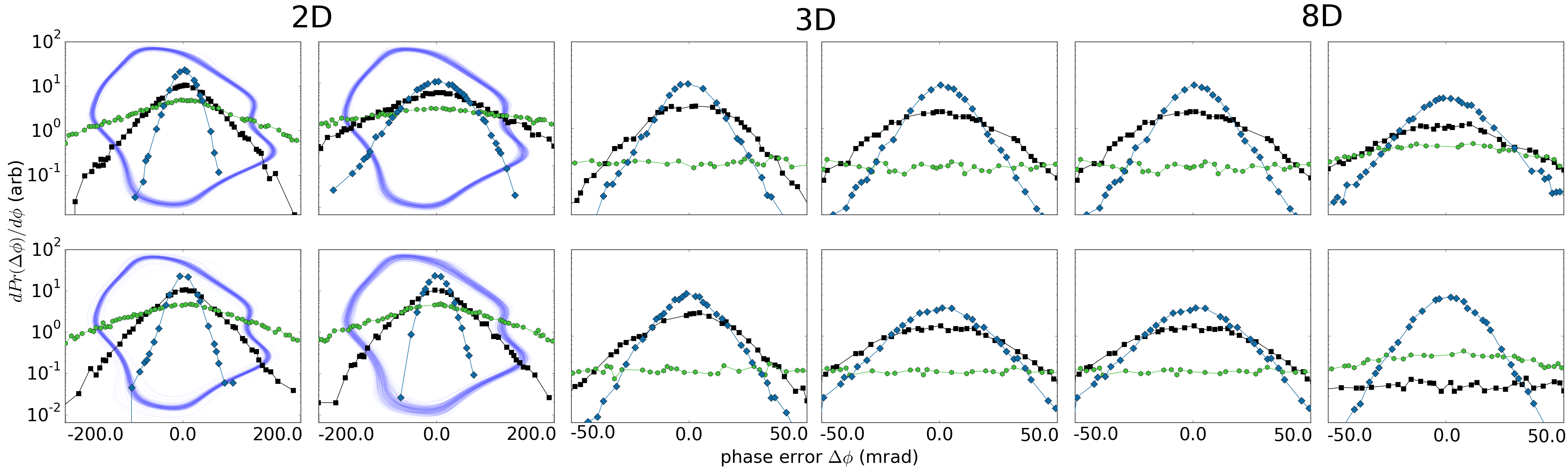}
 \caption{Comparison of phase estimators on 2D, 3D, and 8D systems. %
 Comparison of event phase (green) \Phaser\ (black) and form phase (turquoise) phase noise distributions on simulated trajectories. %
 Plots show four different conditions for each dimension: one baseline condition (top left), an increase in the magnitude of the stochastic diffusion term (top right),
 an increase in the levels of variability in the initial conditions (bottom left) and an increase in noise on the coordinate corresponding to phase in the equivalent deterministic system (bottom right; see \refTbl{simperformance}). %
}
 \label{fig:8dperformance}
\label{fig:2dperformance}
\label{fig:3dperformance}
\end{figure*}
\subsection{Selkov and Fitzhugh Nagumo oscillators}
\begin{figure}[h]
 \centering \includegraphics[width=0.9\columnwidth]{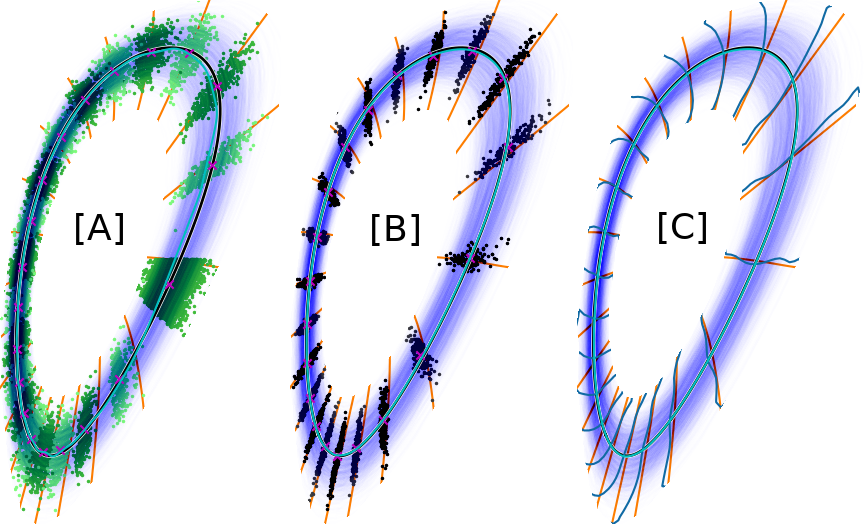}
 \caption{%
 Comparison of stochastically driven Selkov oscillator isochrons obtained from the form-phase, \Phaser\ and event-based phase estimates. %
 We plotted trajectories (blue), the limit cycle (true cycle---black, data driven estimate---teal) and ground-truth isochrons from forward integration (orange lines). %
 On top of those we plotted scatter plots of trajectory points falling in 20 equally spaced intervals of width 0.02 period; ([A], green, alternating shades), \Phaser\ ([B], black), and form-based phase estimates ([C], gray). %
 Event phase estimates become noticeably poorer away from event; \Phaser\ estimates remain equally tight, but their angle to the limit cycle---the PRC---is incorrect; Form phase isochrons closely match the forward integration \concept{ground truth}, that is to say the angle between the thin lines in [C] are more acute than the angle between the black dots and orange lines in [B])
 }
 \label{fig:isochroncomparison}
\end{figure}
\begin{figure*}[h]
  \centering \includegraphics[width=0.9\textwidth]{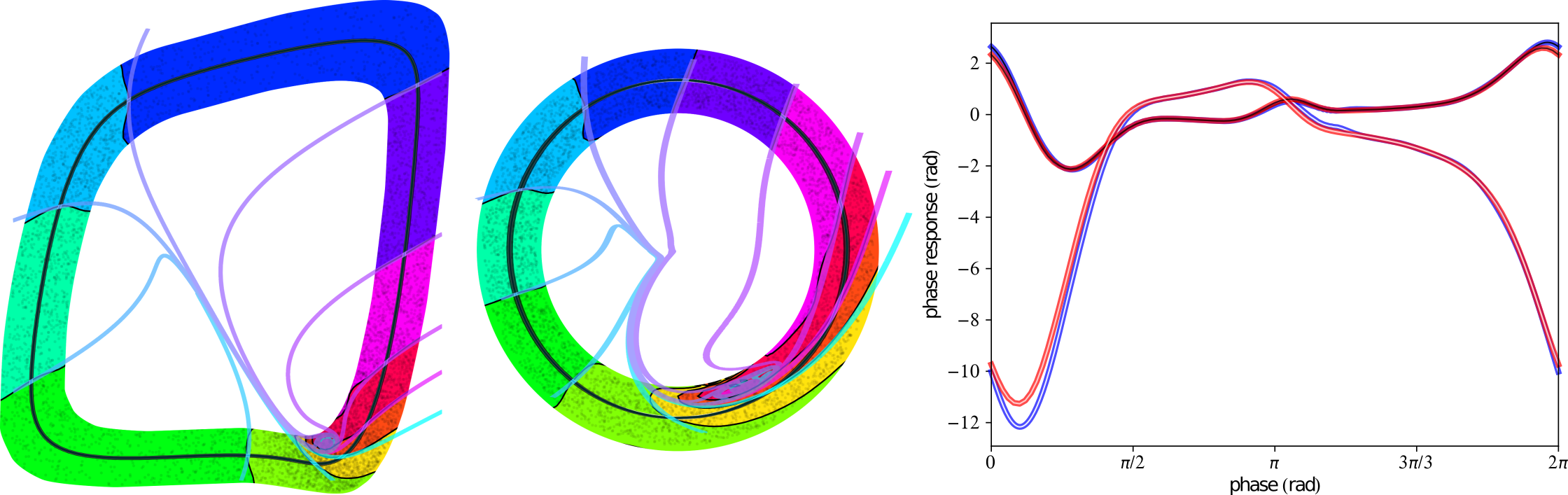}
  \caption{%
    \ToF\ based estimation of the FitzHugh-Nagumo oscillator isochrons (colored bands and solid black curves transverse to cycle). %
    We uniformly sampled points within an annulus around the limit cycle in the rectified coordinates ($6000$ points, black dots).
    We then estimated the \ToF\ with polynomial terms up to order six, and Fourier terms of order up to ten.
    For comparison, we chose the same parameters as fig. 4 from \cite{langfield2014solving} and overlaid their isochrons (pastel lines) on ours (black lines between colored regions).
    The results are shown in the original coordinates (left; limit cycle in black), and in rectified coordinates (middle). %
    The phase response curves (right) for this system as estimated by form phase (blue) and by forward numerical integration (in the sense used in \citet{wilshin2021phase}; red) for both $x$ (white line interior) and $y$ (black line interior) perturbations. %
    Adapted from \cite{langfield2014solving}, with the permission of AIP Publishing.
  }	\label{fig:fhncomparison}
\end{figure*}

One of the particular strengths of our algorithm is that it inherently estimates the phase response curves (PRC \citep{ermentrout1996type, wilshin2021phase}) of the system.
We tested this capability with the Selkov oscillator (\refFig{isochroncomparison}), and a classical neuronal oscillator model---the FitzHugh-Nagumo (\cite{fitzhugh1961impulses, nagumo1962active}, \refFig{fhncomparison}) system.

For Selkov, we took the two-dimensional system of the previous section and estimated the isochrons from three methods of phase estimation: event-based, taking phase to be the fraction of time between events; the \Phaser\ algorithm \cite{RevGuk08}; and the form-phase based algorithm.
We focused on the case of low system noise and highest initial condition noise, which corresponds closely with deterministic dynamics whose state space has been fully sampled.
We found that event-based estimates of the ground-truth isochrons are poor, with a large spread in position and poor agreement with the ground truth.
\Phaser\ performed well on the limit cycle, but the angle of the limit cycle to the estimated isochrons was incorrect, indicating an erroneous PRC estimate.
The form-phase estimate matched the deterministic isochrons far more closely.
We conclude that form-phase based phase estimates provided a superior estimate of the PRC.

Form-phase based estimates can also easily be used for obtaining isochrons of systems for which the equations are known, by using an importance sampling scheme.
\RefFig{fhncomparison} demonstrates this approach applied to the FitzHugh-Nagumo oscillator, and compares our results with those of \cite{langfield2015forward},  which is a state of the art method for systems with known equations of motion.
FitzHugh-Nagumo is particularly challenging as it is a \concept{relaxation oscillator} which produces rapidly changing spikes with gaps of inactivity between them in a neuron-like spike train.
Relaxation oscillators inherently have large changes in the magnitude of $\dT$ on the limit cycle, exacerbating the effects of numerical errors on the estimation of isochrons.
The results demonstrate that we can reproduce the isochrons next to the limit cycle, and thus correctly reproduce the PRC for the FitzHugh-Nagumo oscillator when the complete state is available for measurement, or when the conditions of \citet{wilshin2021phase} are met.
This suggests we could reasonably expect our new algorithm to recover the PRC of neuronal oscillators when provided with sufficient data.

\subsection{Phase estimation from partial data in guineafowl}
\label{sec:resultsanimal}
Perhaps the most significant feature of our algorithm for its intended users is that our algorithm requires only state and state velocity pairs to calculate a phase estimate.
It can be trained using short, disjointed training examples---even if no example contains more than a small portion of a cycle (Theorem~\ref{th:estimated_isochrons_approach_true_isochrons} in \refSupp{tof_from_uncertain}).
This is not true of any event-based phase, since the phase of all segments that do not contain the event cannot be determined at all.
Similarly, \Phaser\ cannot determine the phase of any data coming from a segment that does not cross the phase zero Poincar\'e section.
Furthermore, because it uses the Hilbert transform to produce a protophase, \Phaser\ requires each time series in its training data to be multiple cycles long---otherwise the Hilbert transform exhibits ringing artifacts.

We demonstrate the ability of our \ToF\ algorithm to recover phase from short segments of data in \refFig{cutperformancecomp}.
Not only does it recover phase from input that only contains trajectory segments from partial cycles, the new algorithm also produces a similar result even when some of the data is omitted systematically from a large part of space.
This dataset with systematic omissions does not meet the requirements of our estimator convergence result (Theorem~\ref{th:estimated_isochrons_approach_true_isochrons} in \refSupp{tof_from_uncertain})---roughly speaking, it does not have a uniquely defined phase because a portion of the limit cycle is not observed---but it still seems to be usable, for reasons which remain to be understood (cf. Remark~\ref{rem:noise-assump-generality} in \refSupp{tof_from_uncertain}).
\begin{figure}[h]
  \centering
  \includegraphics[width=0.90\columnwidth]{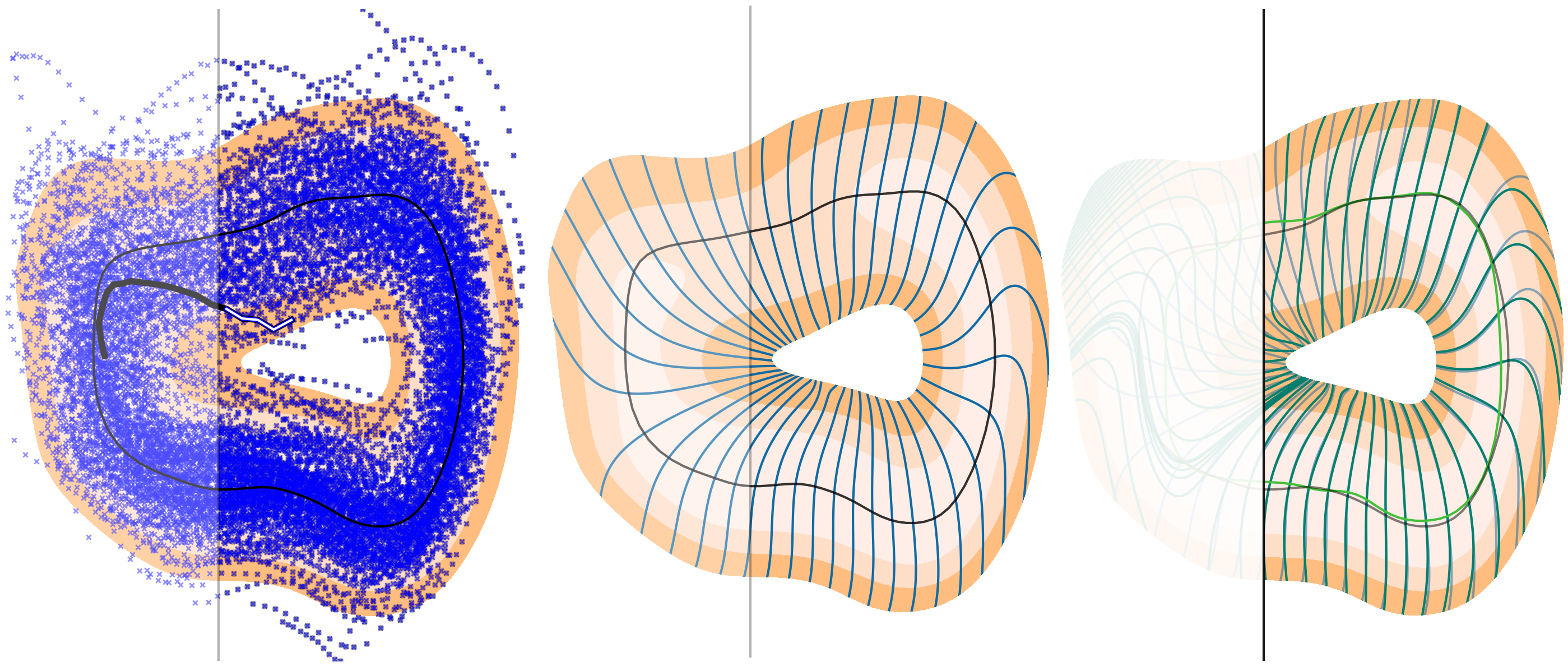}
  \caption{Estimating phase from partial data.
  Results for guineafowl foot data, post-processed adversely. %
  We randomly split \textit{Numida meleagris} foot motions (292 cycles; $250\text{sps}$; running at $3.\text{Hz}$) into $20$ sample segments with gaps $>40$ samples in between. %
  Segments were far shorter than cycles ([A] blue dots all data; example segment wide solid green). %
  We used the \ToF\ phase estimator (order six Fourier, and order six polynomial terms), plotted the limit cycle estimate (order 10 Fourier model; [B], solid black) and isochrons ([B] light teal, every $\pi/20$ radians). %
  We removed all data with $x<-0.5$ (arbitrary units; data is z-scored), about one third of every cycle ([A,B,C] vertical line; omitted data faded on left), and recomputed phase (order six Fourier terms; order 6 polynomial terms), and plotted the limit cycle ([C] solid green; same method) and the same isochrons ([C] dark teal; full data isochrons faded light teal). %
  We indicated amount of training data using a kernel smoothed density plot (50\%, 70\% and 90\% of max density; strong orange to pale orange). %
  Even with fractional cycles and systematic inability to observe a large fraction of the cycle, most structure is recovered in both treatments wherever $>70\%$ of max data density is available. %
  We remark that, while accuracy of the isochron estimates in [B] is to be expected from Theorem~\ref{th:estimated_isochrons_approach_true_isochrons} (\refSupp{tof_from_uncertain}), it remains to be understood why reasonable accuracy is obtained in [C] (cf. Remark~\ref{rem:noise-assump-generality} in \refSupp{tof_from_uncertain}).
  }
\label{fig:cutperformancecomp}
\end{figure}

\subsection{Application to chemical oscillators}
To illustrate how our algorithm could be applied to data from diverse sources, we obtained a sample of chemical oscillator data %
\footnote{Our thanks for these data go to Seth Fraden and Michael Norton.}
.
The raw measurements are seen in \refFig{chemRaw}.
In these relaxation oscillators the observable state is nearly constant away from their rapid ``spike'', and so it is difficult to extract useful state information far from these spikes.
To provide a better numerically conditioned state, after renormalizing the oscillations to account for baseline drift and reagent depletion (\refFig{chemLags} top), we passed the time series through a linear filter bank.
We passed each signal through three Butterworth lowpass filters of order 2, with cutoffs computed from the median inter-spike interval (ISI) as half ISI, ISI, and twice ISI.
We then used the difference of the first two as one set of ``lagged'' signals (\refFig{chemLags} middle), and the difference of the last two as a second set of ``lagged'' signals (\refFig{chemLags} bottom).
These lagged signals have the property of being fairly sinusoidal,  possessing phase responses close to $\pi/2$ apart, and having very little response at ``DC'' (frequency 0).
In short, they appear to provide a good augmented state-space for the phase estimator to operate on (and it seems likely that this intuitive statement can be made mathematically rigorous using methods similar to those of \citet[Sec.~3]{sauer1991embedology}).
We then took the resulting phase estimates as real numbers, unwrapped them, and subtracted their time-dependent mean to obtain the relative phases \citep{Revzen-TestArch07}, showing clearly the convergence into two pairs of anti-phase oscillators (\refFig{chemSync}).
\begin{figure}[h]
  \centering
  \includegraphics[width=0.95\columnwidth]{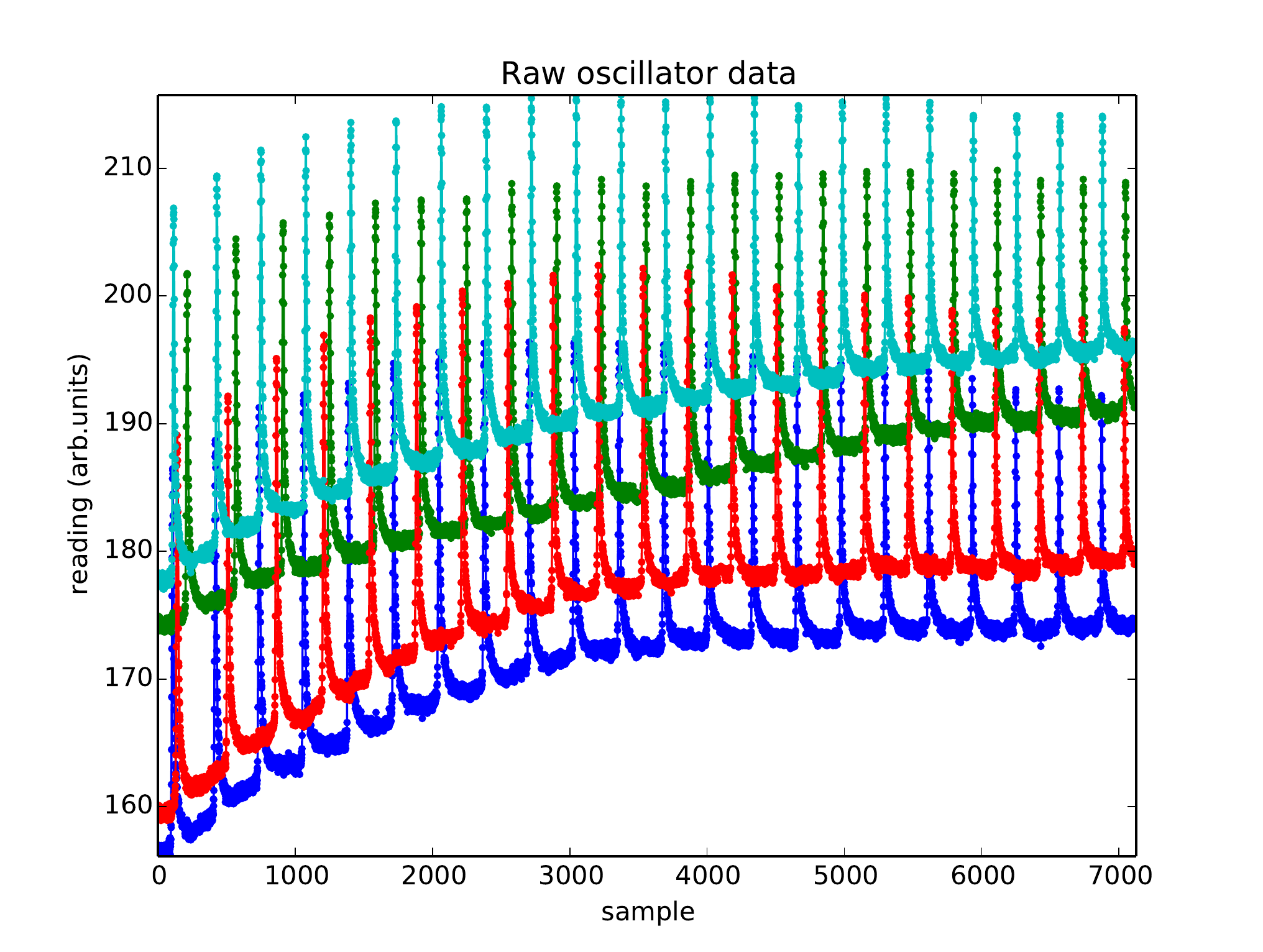}
  \caption{%
  Raw data of optical transmission through four coupled cells running a   Belousov-Zhabotinsky reaction.
  Physical units are not provided as phase estimation does not require it.
 }
\label{fig:chemRaw}
\end{figure}

\begin{figure}[h]
  \centering
  \includegraphics[width=0.95\columnwidth]{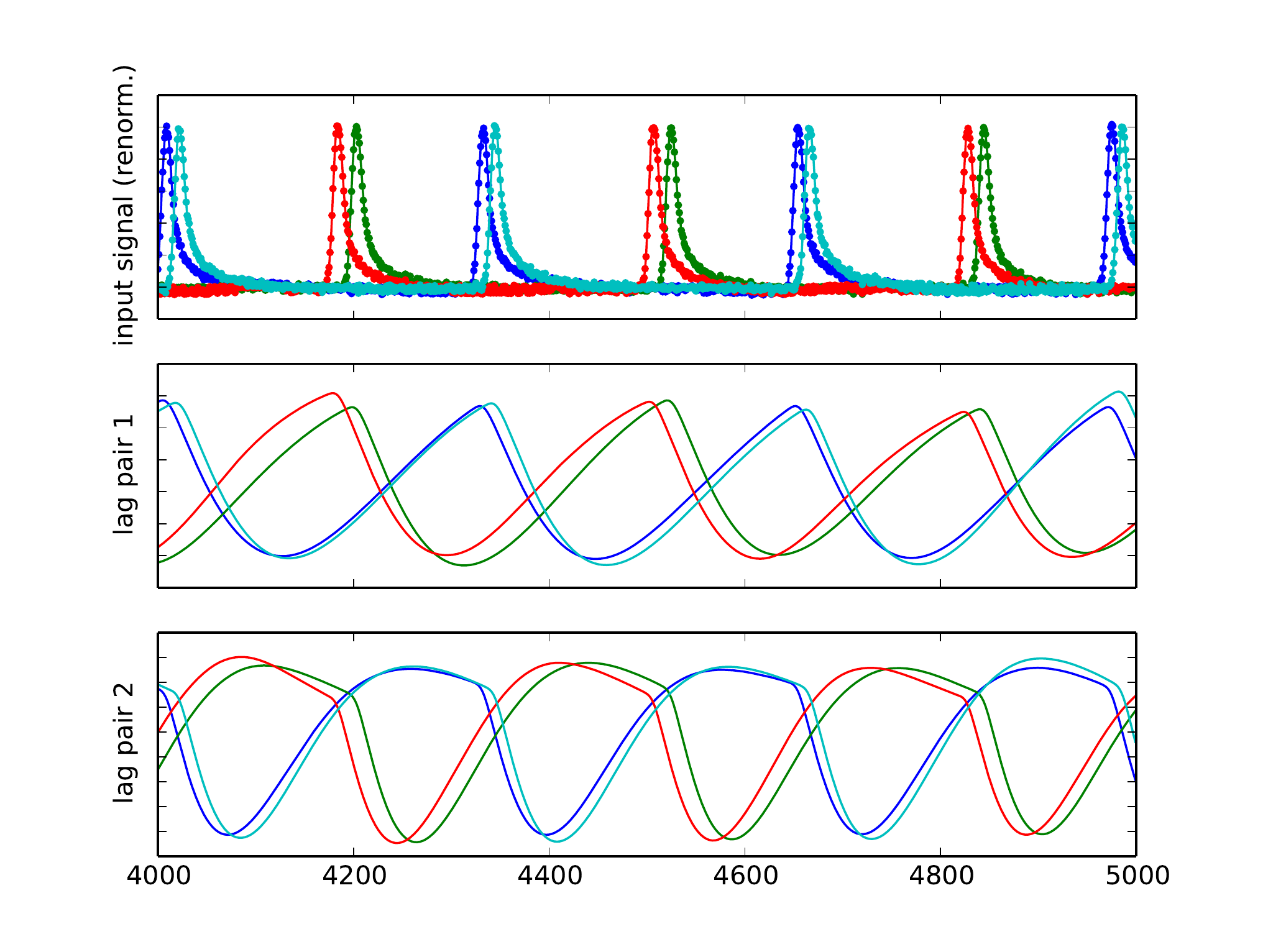}
  \caption{%
  Data preparation of chemical oscillator data.
  [top] detrended and rescaled time series. %
  We then computed the median inter-spike interval and filtered at 2 ISI (hereon s1), 1 ISI (s2), and 0.5 ISI (s3).
  [middle] the difference s1-s2; [bottom] the difference s2-s3
  }
\label{fig:chemLags}
\end{figure}

\begin{figure}[h]
  \centering
  \includegraphics[width=0.95\columnwidth]{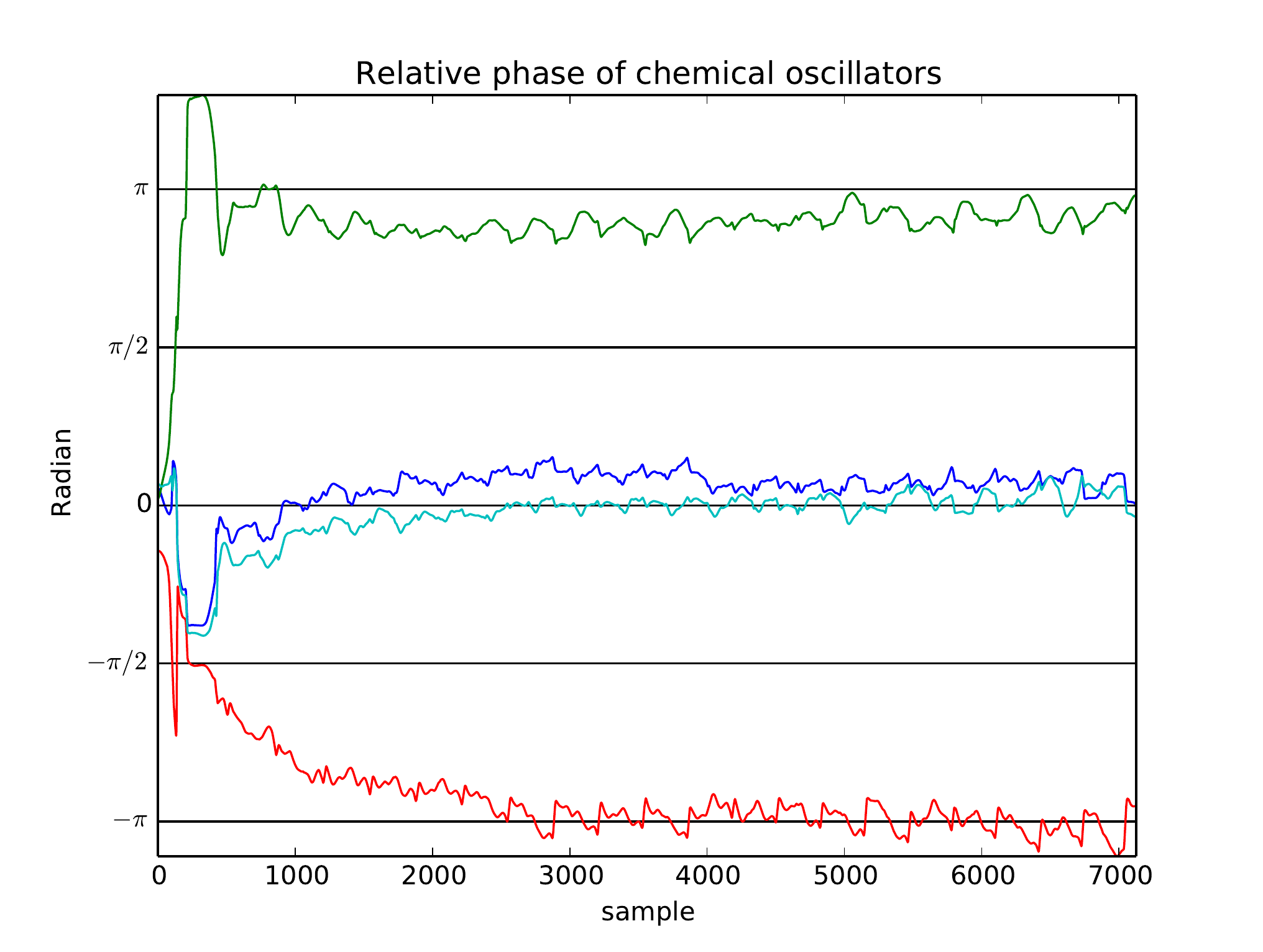}
  \caption{%
  Relative phases of the oscillators in \refFig{chemRaw}, taken by training a phase estimator on the data, and subtracting the mean phase of all the oscillators.
  }
\label{fig:chemSync}
\end{figure}

\section{Conclusion}\label{sec:conclusion}

The definition of the \ToF\ we provided shows that for smooth exponentially stable oscillators the phase obtained from the \ToF\ agrees with the classical definition of asymptotic phase.
Unlike the classical definition, the \ToF\ makes evident (\refSupp{equiv}) that asymptotic phase is in fact uniquely determined by a pointwise condition on forward invariant sets.

This theoretical property allowed us to develop the presented algorithm and its performance guarantees (\refSupp{tof_from_uncertain}) for approximating the \ToF\ of an oscillator from which a sample of noisy trajectory segments is available.
We have shown this algorithm performs well on both simulated and experimental data.
On simulated data of the Fitzhugh-Nagumo oscillator, a challenging relaxation oscillator system with known equations of motion, it performs comparably to the state-of-the-art algorithms which use the equations, even though our algorithm is entirely data driven, and does not require the dynamics themselves to be integrated.

Our algorithm offers significant improvements over previously known methods for phase estimation.
\begin{enumerate}
\item Most importantly, the ability to estimate phase from state and state velocity pairs $(x,\dot x)$.
 This capability in particular would make it attractive for modeling systems for which we are unable to obtain records of full periods of oscillation.
\item The ability to directly obtain phase response curves from experimental time series data.
\item More generally, the ability to estimate phase far from the limit cycle, in any forward invariant region containing enough data (Theorem~\ref{th:estimated_isochrons_approach_true_isochrons} in \refSupp{tof_from_uncertain}).
\item All these abilities come from an estimation procedure with provable convergence and convergence rates under certain assumptions (\refSupp{tof_from_uncertain}).
\end{enumerate}

No other phase estimation method we are aware of possesses all of these important properties.

Because the \ToF\ can be estimated from short time series segments, it offers the opportunity to study a broad range of oscillatory phenomena that were previously hard to analyze---either because of prohibitively long oscillation periods, or because of technical difficulties in obtaining multi-period measurements.
We have shown that a \ToF\ estimator can be successfully trained on short fragments of cycles, and will still estimate the isochrons accurately and consistently (e.g. the oscillator associated with the guineafowl feet).
Event-based methods such as anterior extreme position detection \citep{wnCru88} or heel-strike \citep{ting1994dss, jindrich2002dsr} detection cannot use such fragmented data as the decisive event need not appear in every fragment of data.
The \Phaser\ algorithm cannot solve this problem as it performs a Hilbert transform to calculate its initial proto-phases, introducing transients on the order of 1 to 3 cycles.

The \ToF\ allows one to detect the phase response to perturbations of the underlying deterministic dynamics.
The phase response describes only the effect of perturbation that persists after a long time, and even the infinitesimal phase response approximation---the phase response curve (PRC)---allows the observed system to be modeled in the weakly coupled oscillator approximation.
This PRC is available directly from the \ToF\ phase estimation procedure, paving the way to empirically generated oscillator coupling models in many domains.

For scientists studying oscillatory systems, a phase-based model allows the experimentalist to compare the outcome of experimental treatment (e.g. a mechanical perturbation to locomotion) to that of the counter-factual unperturbed system modeled as a deterministic function of phase \citep{Revzen-TestArch07}.
The results in \refSec{resultssim} indicate that for systems of different dimensions, levels of diffusion, and initial condition noise, approximating phase using a \ToF\ produced by the algorithm presented here is superior to event-based phase estimates commonly used in the experimental literature, and to the \Phaser\ algorithm.
In the examples we studied, the mean-square error of the phase estimate from the form-based technique is smaller.
The inferred isochrons are closer to the ground truth, and the phase response curves are more accurate.

Future work may extend our algorithm by application of domain-specific knowledge, in particular by incorporating new basis forms and coordinate change methods.
Such extensions would certainly speed up the algorithm and improve its accuracy.
Another future direction could include more general large-scale numerical solvers which could be applied to approximating the \ToF\ without resorting to explicit basis function expansions.

The \ToF\ offers a new perspective on the notion of phase in nonlinear oscillators, as well as a new way to obtain oscillator and PRC models from experimental data.

\subsection*{Acknowledgements}
We are deeply indebted to John M. Guckenheimer for conceiving of the \ToF\ and providing guidance and comments.
We also thank Dan Guralnik and G. Bard Ermentrout for valuable insights.
This work was supported by supported by EPSRC grant EP/H04924X/1, BBSRC grant BB/J021504/1, and U. S. Army Research Office under contract/grant number 64929EG supporting S. Wilshin,
ARO W911NF-18-1-0327 and ONR N000141612817 (held by D. E. Koditschek) supporting M. D. Kvalheim,
NSF Grants 1422157 and 1838179 to C. Scott, and ARO Young Investigator Award \#61770 and W911NF-14-1-0573 to S. Revzen.
The guineafowl kinematics were gathered under the supervision of M. Daley with the support of a Human Frontier Science Program grant (RGY0062/2010) and ethical approval was obtained from Royal Veterinary College Ethics and Welfare Committee under the protocol title, ``Kinematics and kinetics in birds running over uneven terrain''.
The optical transmission data for the Belousov-Zhabotinsky reaction were obtained from S. Fraden and M. Norton of Brandeis University, under awards from the National Science Foundation NSF DMREF-1534890 and U. S. Army Research Laboratory and U. S. Army Research Office W911NF-16-1-0094.

\bibliography{all}

\begin{thebibliography}{95}%
\makeatletter
\providecommand \@ifxundefined [1]{%
 \@ifx{#1\undefined}
}%
\providecommand \@ifnum [1]{%
 \ifnum #1\expandafter \@firstoftwo
 \else \expandafter \@secondoftwo
 \fi
}%
\providecommand \@ifx [1]{%
 \ifx #1\expandafter \@firstoftwo
 \else \expandafter \@secondoftwo
 \fi
}%
\providecommand \natexlab [1]{#1}%
\providecommand \enquote  [1]{``#1''}%
\providecommand \bibnamefont  [1]{#1}%
\providecommand \bibfnamefont [1]{#1}%
\providecommand \citenamefont [1]{#1}%
\providecommand \href@noop [0]{\@secondoftwo}%
\providecommand \href [0]{\begingroup \@sanitize@url \@href}%
\providecommand \@href[1]{\@@startlink{#1}\@@href}%
\providecommand \@@href[1]{\endgroup#1\@@endlink}%
\providecommand \@sanitize@url [0]{\catcode `\\12\catcode `\$12\catcode
  `\&12\catcode `\#12\catcode `\^12\catcode `\_12\catcode `\%12\relax}%
\providecommand \@@startlink[1]{}%
\providecommand \@@endlink[0]{}%
\providecommand \url  [0]{\begingroup\@sanitize@url \@url }%
\providecommand \@url [1]{\endgroup\@href {#1}{\urlprefix }}%
\providecommand \urlprefix  [0]{URL }%
\providecommand \Eprint [0]{\href }%
\providecommand \doibase [0]{http://dx.doi.org/}%
\providecommand \selectlanguage [0]{\@gobble}%
\providecommand \bibinfo  [0]{\@secondoftwo}%
\providecommand \bibfield  [0]{\@secondoftwo}%
\providecommand \translation [1]{[#1]}%
\providecommand \BibitemOpen [0]{}%
\providecommand \bibitemStop [0]{}%
\providecommand \bibitemNoStop [0]{.\EOS\space}%
\providecommand \EOS [0]{\spacefactor3000\relax}%
\providecommand \BibitemShut  [1]{\csname bibitem#1\endcsname}%
\let\auto@bib@innerbib\@empty
\bibitem [{\citenamefont {Revzen}\ and\ \citenamefont
  {Kvalheim}(2015)}]{revzen2015data}%
  \BibitemOpen
  \bibfield  {author} {\bibinfo {author} {\bibfnamefont {S.}~\bibnamefont
  {Revzen}}\ and\ \bibinfo {author} {\bibfnamefont {M.}~\bibnamefont
  {Kvalheim}},\ }in\ \href {\doibase 10.1117/12.2178007} {\emph {\bibinfo
  {booktitle} {Micro-and Nanotechnology Sensors, Systems, and Applications
  VII}}},\ Vol.\ \bibinfo {volume} {9467}\ (\bibinfo {organization}
  {International Society for Optics and Photonics},\ \bibinfo {year} {2015})\
  p.\ \bibinfo {pages} {94671V}\BibitemShut {NoStop}%
\bibitem [{\citenamefont {Seipel}\ \emph {et~al.}(2017)\citenamefont {Seipel},
  \citenamefont {Kvalheim}, \citenamefont {Revzen}, \citenamefont {Sharbafi},\
  and\ \citenamefont {Seyfarth}}]{seipel2017conceptual}%
  \BibitemOpen
  \bibfield  {author} {\bibinfo {author} {\bibfnamefont {J.}~\bibnamefont
  {Seipel}}, \bibinfo {author} {\bibfnamefont {M.}~\bibnamefont {Kvalheim}},
  \bibinfo {author} {\bibfnamefont {S.}~\bibnamefont {Revzen}}, \bibinfo
  {author} {\bibfnamefont {M.~A.}\ \bibnamefont {Sharbafi}}, \ and\ \bibinfo
  {author} {\bibfnamefont {A.}~\bibnamefont {Seyfarth}},\ }in\ \href@noop {}
  {\emph {\bibinfo {booktitle} {Bioinspired Legged Locomotion}}}\ (\bibinfo
  {publisher} {Elsevier},\ \bibinfo {year} {2017})\ pp.\ \bibinfo {pages}
  {55--131}\BibitemShut {NoStop}%
\bibitem [{\citenamefont {Kvalheim}\ and\ \citenamefont
  {Bloch}(2021)}]{kvalheim2021families}%
  \BibitemOpen
  \bibfield  {author} {\bibinfo {author} {\bibfnamefont {M.~D.}\ \bibnamefont
  {Kvalheim}}\ and\ \bibinfo {author} {\bibfnamefont {A.~M.}\ \bibnamefont
  {Bloch}},\ }\href {\doibase 10.1016/j.jde.2021.03.009} {\bibfield  {journal}
  {\bibinfo  {journal} {J. Differential Equations}\ }\textbf {\bibinfo {volume}
  {285}},\ \bibinfo {pages} {211} (\bibinfo {year} {2021})}\BibitemShut
  {NoStop}%
\bibitem [{\citenamefont {Hoppensteadt}\ and\ \citenamefont
  {Izhikevich}(1997)}]{hoppensteadt2012weakly}%
  \BibitemOpen
  \bibfield  {author} {\bibinfo {author} {\bibfnamefont {F.~C.}\ \bibnamefont
  {Hoppensteadt}}\ and\ \bibinfo {author} {\bibfnamefont {E.~M.}\ \bibnamefont
  {Izhikevich}},\ }\href {\doibase 10.1007/978-1-4612-1828-9} {\emph {\bibinfo
  {title} {Weakly connected neural networks}}},\ Vol.\ \bibinfo {volume} {126}\
  (\bibinfo  {publisher} {Springer-Verlag},\ \bibinfo {year}
  {1997})\BibitemShut {NoStop}%
\bibitem [{\citenamefont {May}(1972)}]{may1972limit}%
  \BibitemOpen
  \bibfield  {author} {\bibinfo {author} {\bibfnamefont {R.~M.}\ \bibnamefont
  {May}},\ }\href@noop {} {\bibfield  {journal} {\bibinfo  {journal} {Science}\
  }\textbf {\bibinfo {volume} {177}},\ \bibinfo {pages} {900} (\bibinfo {year}
  {1972})}\BibitemShut {NoStop}%
\bibitem [{\citenamefont {Epstein}\ and\ \citenamefont
  {Pojman}(1998)}]{epstein1998introduction}%
  \BibitemOpen
  \bibfield  {author} {\bibinfo {author} {\bibfnamefont {I.~R.}\ \bibnamefont
  {Epstein}}\ and\ \bibinfo {author} {\bibfnamefont {J.~A.}\ \bibnamefont
  {Pojman}},\ }\href@noop {} {\emph {\bibinfo {title} {An introduction to
  nonlinear chemical dynamics: oscillations, waves, patterns, and chaos}}}\
  (\bibinfo  {publisher} {Oxford University Press},\ \bibinfo {year}
  {1998})\BibitemShut {NoStop}%
\bibitem [{\citenamefont {G{\l}azek}\ and\ \citenamefont
  {Wilson}(2002)}]{glazek2002limit}%
  \BibitemOpen
  \bibfield  {author} {\bibinfo {author} {\bibfnamefont {S.~D.}\ \bibnamefont
  {G{\l}azek}}\ and\ \bibinfo {author} {\bibfnamefont {K.~G.}\ \bibnamefont
  {Wilson}},\ }\href@noop {} {\bibfield  {journal} {\bibinfo  {journal}
  {Physical review letters}\ }\textbf {\bibinfo {volume} {89}},\ \bibinfo
  {pages} {230401} (\bibinfo {year} {2002})}\BibitemShut {NoStop}%
\bibitem [{\citenamefont {Van~der Pol}(1934)}]{van1934nonlinear}%
  \BibitemOpen
  \bibfield  {author} {\bibinfo {author} {\bibfnamefont {B.}~\bibnamefont
  {Van~der Pol}},\ }\href@noop {} {\bibfield  {journal} {\bibinfo  {journal}
  {Proceedings of the Institute of Radio Engineers}\ }\textbf {\bibinfo
  {volume} {22}},\ \bibinfo {pages} {1051} (\bibinfo {year}
  {1934})}\BibitemShut {NoStop}%
\bibitem [{\citenamefont {Duffing}(1918)}]{duffing1918erzwungene}%
  \BibitemOpen
  \bibfield  {author} {\bibinfo {author} {\bibfnamefont {G.}~\bibnamefont
  {Duffing}},\ }\href@noop {} {\emph {\bibinfo {title} {Erzwungene Schwingungen
  bei ver{\"a}nderlicher Eigenfrequenz und ihre technische Bedeutung}}},\
  \bibinfo {number} {41-42}\ (\bibinfo  {publisher} {R, Vieweg \& Sohn},\
  \bibinfo {year} {1918})\BibitemShut {NoStop}%
\bibitem [{\citenamefont {Stoker}(1950)}]{stoker1950nonlinear}%
  \BibitemOpen
  \bibfield  {author} {\bibinfo {author} {\bibfnamefont {J.~J.}\ \bibnamefont
  {Stoker}},\ }\href@noop {} {\emph {\bibinfo {title} {Nonlinear vibrations in
  mechanical and electrical systems}}},\ Vol.~\bibinfo {volume} {2}\ (\bibinfo
  {publisher} {Interscience Publishers New York},\ \bibinfo {year}
  {1950})\BibitemShut {NoStop}%
\bibitem [{\citenamefont {FitzHugh}(1961)}]{fitzhugh1961impulses}%
  \BibitemOpen
  \bibfield  {author} {\bibinfo {author} {\bibfnamefont {R.}~\bibnamefont
  {FitzHugh}},\ }\href@noop {} {\bibfield  {journal} {\bibinfo  {journal}
  {Biophysical journal}\ }\textbf {\bibinfo {volume} {1}},\ \bibinfo {pages}
  {445} (\bibinfo {year} {1961})}\BibitemShut {NoStop}%
\bibitem [{\citenamefont {Nagumo}\ \emph {et~al.}(1962)\citenamefont {Nagumo},
  \citenamefont {Arimoto},\ and\ \citenamefont {Yoshizawa}}]{nagumo1962active}%
  \BibitemOpen
  \bibfield  {author} {\bibinfo {author} {\bibfnamefont {J.}~\bibnamefont
  {Nagumo}}, \bibinfo {author} {\bibfnamefont {S.}~\bibnamefont {Arimoto}}, \
  and\ \bibinfo {author} {\bibfnamefont {S.}~\bibnamefont {Yoshizawa}},\
  }\href@noop {} {\bibfield  {journal} {\bibinfo  {journal} {Proceedings of the
  IRE}\ }\textbf {\bibinfo {volume} {50}},\ \bibinfo {pages} {2061} (\bibinfo
  {year} {1962})}\BibitemShut {NoStop}%
\bibitem [{\citenamefont {Sel'Kov}(1968)}]{sel1968self}%
  \BibitemOpen
  \bibfield  {author} {\bibinfo {author} {\bibfnamefont {E.~E.}\ \bibnamefont
  {Sel'Kov}},\ }\href@noop {} {\bibfield  {journal} {\bibinfo  {journal} {The
  FEBS Journal}\ }\textbf {\bibinfo {volume} {4}},\ \bibinfo {pages} {79}
  (\bibinfo {year} {1968})}\BibitemShut {NoStop}%
\bibitem [{\citenamefont {Guckenheimer}(1975)}]{shGuc75}%
  \BibitemOpen
  \bibfield  {author} {\bibinfo {author} {\bibfnamefont {J.}~\bibnamefont
  {Guckenheimer}},\ }\href@noop {} {\bibfield  {journal} {\bibinfo  {journal}
  {J Math Biol}\ }\textbf {\bibinfo {volume} {1}},\ \bibinfo {pages} {259}
  (\bibinfo {year} {1975})}\BibitemShut {NoStop}%
\bibitem [{Note1()}]{Note1}%
  \BibitemOpen
  \bibinfo {note} {This is the exterior derivative and for Riemannian spaces
  when acting on a scalar is the operator $\nabla $ with its contravariant
  index lowered by the metric. Since our method is not metric dependent we use
  $\protect \mathrm {\protect \mathbf {d}}{}$ here, however readers unfamiliar
  with differential geometry can treat $\protect \mathrm {\protect \mathbf
  {d}}{}$ as $\nabla $.}\BibitemShut {Stop}%
\bibitem [{Note2()}]{Note2}%
  \BibitemOpen
  \bibinfo {note} {While it need not, it could; e.g. $\protect \dot \theta = 1
  + c r; \protect \dot r = (r-2)(r-1)r$ has phase identically $\theta $ for
  $c=0$, but phase diverges at $r=2$ when $c\protect \neq 0$.}\BibitemShut
  {Stop}%
\bibitem [{\citenamefont {Revzen}\ and\ \citenamefont
  {Guckenheimer}(2008)}]{RevGuk08}%
  \BibitemOpen
  \bibfield  {author} {\bibinfo {author} {\bibfnamefont {S.}~\bibnamefont
  {Revzen}}\ and\ \bibinfo {author} {\bibfnamefont {J.~M.}\ \bibnamefont
  {Guckenheimer}},\ }\href {\doibase 10.1103/PhysRevE.78.051907} {\bibfield
  {journal} {\bibinfo  {journal} {Phys Rev E}\ }\textbf {\bibinfo {volume}
  {78}},\ \bibinfo {pages} {051907} (\bibinfo {year} {2008})}\BibitemShut
  {NoStop}%
\bibitem [{Note3()}]{Note3}%
  \BibitemOpen
  \bibinfo {note} {We use the term ``{basis}'' informally, as it is used e.g.
  in the notion of ``{radial basis functions}'' in machine learning, to mean a
  collection of functions whose finite linear spans are dense in the function
  space of interest.}\BibitemShut {Stop}%
\bibitem [{Note4()}]{Note4}%
  \BibitemOpen
  \bibinfo {note} {Although the closed 1-form we seek is not exact, algebraic
  topology teaches us that, on the basin of attraction associated to an
  oscillator, any two closed 1-forms $\alpha ,b$ can be related as $a=\omega b
  +c$ for some $\omega $ real and $c$ exact. Exact forms are exterior
  derivatives (``gradients'') of scalar functions. The authors are deeply
  indebted to Dan Guralnik for pointing out this insight on the de~Rham
  cohomology of oscillators. That this insight extends from $\protect \mathcal
  {C}^\infty $ closed forms (used in de~Rham cohomology) to $\protect \mathcal
  {C}^0$ closed forms follows from general results on smoothing currents \cite
  [pp. 61-70]{deRham1984differentiable}; alternatively, it is not difficult to
  give a direct proof of this result for oscillators.}\BibitemShut {Stop}%
\bibitem [{\citenamefont {Danner}\ \emph {et~al.}(2015)\citenamefont {Danner},
  \citenamefont {Hofstoetter}, \citenamefont {Freundl}, \citenamefont {Binder},
  \citenamefont {Mayr}, \citenamefont {Rattay},\ and\ \citenamefont
  {Minassian}}]{danner2015human}%
  \BibitemOpen
  \bibfield  {author} {\bibinfo {author} {\bibfnamefont {S.~M.}\ \bibnamefont
  {Danner}}, \bibinfo {author} {\bibfnamefont {U.~S.}\ \bibnamefont
  {Hofstoetter}}, \bibinfo {author} {\bibfnamefont {B.}~\bibnamefont
  {Freundl}}, \bibinfo {author} {\bibfnamefont {H.}~\bibnamefont {Binder}},
  \bibinfo {author} {\bibfnamefont {W.}~\bibnamefont {Mayr}}, \bibinfo {author}
  {\bibfnamefont {F.}~\bibnamefont {Rattay}}, \ and\ \bibinfo {author}
  {\bibfnamefont {K.}~\bibnamefont {Minassian}},\ }\href@noop {} {\bibfield
  {journal} {\bibinfo  {journal} {Brain}\ }\textbf {\bibinfo {volume} {138}},\
  \bibinfo {pages} {577} (\bibinfo {year} {2015})}\BibitemShut {NoStop}%
\bibitem [{\citenamefont {Ting}\ \emph {et~al.}(1994)\citenamefont {Ting},
  \citenamefont {Blickhan},\ and\ \citenamefont {Full}}]{ting1994dss}%
  \BibitemOpen
  \bibfield  {author} {\bibinfo {author} {\bibfnamefont {L.~H.}\ \bibnamefont
  {Ting}}, \bibinfo {author} {\bibfnamefont {R.}~\bibnamefont {Blickhan}}, \
  and\ \bibinfo {author} {\bibfnamefont {R.~J.}\ \bibnamefont {Full}},\
  }\href@noop {} {\bibfield  {journal} {\bibinfo  {journal} {J Exp Biol}\
  }\textbf {\bibinfo {volume} {197}},\ \bibinfo {pages} {251} (\bibinfo {year}
  {1994})}\BibitemShut {NoStop}%
\bibitem [{\citenamefont {Jindrich}\ and\ \citenamefont
  {Full}(2002)}]{jindrich2002dsr}%
  \BibitemOpen
  \bibfield  {author} {\bibinfo {author} {\bibfnamefont {D.~L.}\ \bibnamefont
  {Jindrich}}\ and\ \bibinfo {author} {\bibfnamefont {R.~J.}\ \bibnamefont
  {Full}},\ }\href@noop {} {\bibfield  {journal} {\bibinfo  {journal} {J Exp
  Biol}\ }\textbf {\bibinfo {volume} {205}},\ \bibinfo {pages} {2803} (\bibinfo
  {year} {2002})}\BibitemShut {NoStop}%
\bibitem [{\citenamefont {Cruse}\ and\ \citenamefont
  {Schwarze}(1988)}]{wnCru88}%
  \BibitemOpen
  \bibfield  {author} {\bibinfo {author} {\bibfnamefont {H.}~\bibnamefont
  {Cruse}}\ and\ \bibinfo {author} {\bibfnamefont {W.}~\bibnamefont
  {Schwarze}},\ }\href@noop {} {\bibfield  {journal} {\bibinfo  {journal} {J
  Exp Biol}\ }\textbf {\bibinfo {volume} {138}},\ \bibinfo {pages} {455}
  (\bibinfo {year} {1988})}\BibitemShut {NoStop}%
\bibitem [{\citenamefont {Langfield}\ \emph {et~al.}(2014)\citenamefont
  {Langfield}, \citenamefont {Krauskopf},\ and\ \citenamefont
  {Osinga}}]{langfield2014solving}%
  \BibitemOpen
  \bibfield  {author} {\bibinfo {author} {\bibfnamefont {P.}~\bibnamefont
  {Langfield}}, \bibinfo {author} {\bibfnamefont {B.}~\bibnamefont
  {Krauskopf}}, \ and\ \bibinfo {author} {\bibfnamefont {H.~M.}\ \bibnamefont
  {Osinga}},\ }\href@noop {} {\bibfield  {journal} {\bibinfo  {journal} {Chaos:
  An Interdisciplinary Journal of Nonlinear Science}\ }\textbf {\bibinfo
  {volume} {24}},\ \bibinfo {pages} {013131} (\bibinfo {year}
  {2014})}\BibitemShut {NoStop}%
\bibitem [{\citenamefont {Wilshin}\ \emph {et~al.}(2021)\citenamefont
  {Wilshin}, \citenamefont {Kvalheim},\ and\ \citenamefont
  {Revzen}}]{wilshin2021phase}%
  \BibitemOpen
  \bibfield  {author} {\bibinfo {author} {\bibfnamefont {S.}~\bibnamefont
  {Wilshin}}, \bibinfo {author} {\bibfnamefont {M.~D.}\ \bibnamefont
  {Kvalheim}}, \ and\ \bibinfo {author} {\bibfnamefont {S.}~\bibnamefont
  {Revzen}},\ }\href@noop {} {\enquote {\bibinfo {title} {Phase response curves
  and the role of coordinates},}\ } (\bibinfo {year} {2021}),\ \Eprint
  {http://arxiv.org/abs/2111.06511} {arXiv:2111.06511 [q-bio.QM]} \BibitemShut
  {NoStop}%
\bibitem [{\citenamefont {Ermentrout}(1996)}]{ermentrout1996type}%
  \BibitemOpen
  \bibfield  {author} {\bibinfo {author} {\bibfnamefont {B.}~\bibnamefont
  {Ermentrout}},\ }\href@noop {} {\bibfield  {journal} {\bibinfo  {journal}
  {Neural computation}\ }\textbf {\bibinfo {volume} {8}},\ \bibinfo {pages}
  {979} (\bibinfo {year} {1996})}\BibitemShut {NoStop}%
\bibitem [{\citenamefont {Langfield}\ \emph {et~al.}(2015)\citenamefont
  {Langfield}, \citenamefont {Krauskopf},\ and\ \citenamefont
  {Osinga}}]{langfield2015forward}%
  \BibitemOpen
  \bibfield  {author} {\bibinfo {author} {\bibfnamefont {P.}~\bibnamefont
  {Langfield}}, \bibinfo {author} {\bibfnamefont {B.}~\bibnamefont
  {Krauskopf}}, \ and\ \bibinfo {author} {\bibfnamefont {H.~M.}\ \bibnamefont
  {Osinga}},\ }\href@noop {} {\bibfield  {journal} {\bibinfo  {journal} {SIAM
  Journal on Applied Dynamical Systems}\ }\textbf {\bibinfo {volume} {14}},\
  \bibinfo {pages} {1418} (\bibinfo {year} {2015})}\BibitemShut {NoStop}%
\bibitem [{Note5()}]{Note5}%
  \BibitemOpen
  \bibinfo {note} {Our thanks for these data go to Seth Fraden and Michael
  Norton.}\BibitemShut {Stop}%
\bibitem [{\citenamefont {Sauer}\ \emph {et~al.}(1991)\citenamefont {Sauer},
  \citenamefont {Yorke},\ and\ \citenamefont {Casdagli}}]{sauer1991embedology}%
  \BibitemOpen
  \bibfield  {author} {\bibinfo {author} {\bibfnamefont {T.}~\bibnamefont
  {Sauer}}, \bibinfo {author} {\bibfnamefont {J.~A.}\ \bibnamefont {Yorke}}, \
  and\ \bibinfo {author} {\bibfnamefont {M.}~\bibnamefont {Casdagli}},\ }\href
  {\doibase 10.1007/BF01053745} {\bibfield  {journal} {\bibinfo  {journal} {J.
  Statist. Phys.}\ }\textbf {\bibinfo {volume} {65}},\ \bibinfo {pages} {579}
  (\bibinfo {year} {1991})}\BibitemShut {NoStop}%
\bibitem [{\citenamefont {Revzen}\ \emph {et~al.}(2008)\citenamefont {Revzen},
  \citenamefont {Koditschek},\ and\ \citenamefont {Full}}]{Revzen-TestArch07}%
  \BibitemOpen
  \bibfield  {author} {\bibinfo {author} {\bibfnamefont {S.}~\bibnamefont
  {Revzen}}, \bibinfo {author} {\bibfnamefont {D.~E.}\ \bibnamefont
  {Koditschek}}, \ and\ \bibinfo {author} {\bibfnamefont {R.~J.}\ \bibnamefont
  {Full}},\ }\enquote {\bibinfo {title} {Progress in motor control - a
  multidisciplinary perspective},}\ \ (\bibinfo  {publisher} {Springer
  Science+Business Media, LLC - NY},\ \bibinfo {year} {2008})\ Chap.\ \bibinfo
  {chapter} {Towards testable neuromechanical control architectures for
  running}, pp.\ \bibinfo {pages} {25--56}\BibitemShut {NoStop}%
\bibitem [{\citenamefont {de~Rham}(1984)}]{deRham1984differentiable}%
  \BibitemOpen
  \bibfield  {author} {\bibinfo {author} {\bibfnamefont {G.}~\bibnamefont
  {de~Rham}},\ }\href@noop {} {\emph {\bibinfo {title} {Differentiable
  manifolds}}}\ (\bibinfo  {publisher} {Springer-Verlag},\ \bibinfo {year}
  {1984})\BibitemShut {NoStop}%
\bibitem [{\citenamefont {Eldering}\ \emph {et~al.}(2018)\citenamefont
  {Eldering}, \citenamefont {Kvalheim},\ and\ \citenamefont
  {Revzen}}]{eldering2018global}%
  \BibitemOpen
  \bibfield  {author} {\bibinfo {author} {\bibfnamefont {J.}~\bibnamefont
  {Eldering}}, \bibinfo {author} {\bibfnamefont {M.}~\bibnamefont {Kvalheim}},
  \ and\ \bibinfo {author} {\bibfnamefont {S.}~\bibnamefont {Revzen}},\ }\href
  {\doibase 10.1088/1361-6544/aaca8d} {\bibfield  {journal} {\bibinfo
  {journal} {Nonlinearity}\ }\textbf {\bibinfo {volume} {31}},\ \bibinfo
  {pages} {4202} (\bibinfo {year} {2018})}\BibitemShut {NoStop}%
\bibitem [{\citenamefont {Lee}(2013)}]{lee2013smooth}%
  \BibitemOpen
  \bibfield  {author} {\bibinfo {author} {\bibfnamefont {J.~M.}\ \bibnamefont
  {Lee}},\ }\href {\doibase 10.1007/978-1-4419-9982-5} {\emph {\bibinfo {title}
  {Introduction to Smooth Manifolds}}},\ \bibinfo {edition} {2nd}\ ed.\
  (\bibinfo  {publisher} {Springer},\ \bibinfo {year} {2013})\BibitemShut
  {NoStop}%
\bibitem [{\citenamefont {Guillemin}\ and\ \citenamefont
  {Pollack}(2010)}]{guillemin1974differential}%
  \BibitemOpen
  \bibfield  {author} {\bibinfo {author} {\bibfnamefont {V.}~\bibnamefont
  {Guillemin}}\ and\ \bibinfo {author} {\bibfnamefont {A.}~\bibnamefont
  {Pollack}},\ }\href {\doibase 10.1090/chel/370} {\emph {\bibinfo {title}
  {Differential topology}}}\ (\bibinfo  {publisher} {AMS Chelsea Publishing,
  Providence, RI},\ \bibinfo {year} {2010})\ pp.\ \bibinfo {pages}
  {xviii+224},\ \bibinfo {note} {reprint of the 1974 original}\BibitemShut
  {NoStop}%
\bibitem [{\citenamefont {Hirsch}(1976)}]{hirsch1976differential}%
  \BibitemOpen
  \bibfield  {author} {\bibinfo {author} {\bibfnamefont {M.~W.}\ \bibnamefont
  {Hirsch}},\ }\href@noop {} {\emph {\bibinfo {title} {Differential
  topology}}}\ (\bibinfo  {publisher} {Springer-Verlag},\ \bibinfo {year}
  {1976})\BibitemShut {NoStop}%
\bibitem [{\citenamefont {Spivak}(1971)}]{spivak1971calculus}%
  \BibitemOpen
  \bibfield  {author} {\bibinfo {author} {\bibfnamefont {M.}~\bibnamefont
  {Spivak}},\ }\href@noop {} {\emph {\bibinfo {title} {Calculus on
  Manifolds}}}\ (\bibinfo  {publisher} {Perseus Books Publishing},\ \bibinfo
  {year} {1971})\BibitemShut {NoStop}%
\bibitem [{\citenamefont {Arnold}(1974)}]{arnold1974stochastic}%
  \BibitemOpen
  \bibfield  {author} {\bibinfo {author} {\bibfnamefont {L.}~\bibnamefont
  {Arnold}},\ }\href@noop {} {\emph {\bibinfo {title} {Stochastic differential
  equations: theory and applications}}},\ \bibinfo {edition} {1st}\ ed.\
  (\bibinfo  {publisher} {John Wiley and Sons},\ \bibinfo {year}
  {1974})\BibitemShut {NoStop}%
\bibitem [{\citenamefont {Gardiner}(2004)}]{gardiner2004stochastic}%
  \BibitemOpen
  \bibfield  {author} {\bibinfo {author} {\bibfnamefont {C.~W.}\ \bibnamefont
  {Gardiner}},\ }\href@noop {} {\emph {\bibinfo {title} {Handbook of Stochastic
  Methods}}},\ \bibinfo {edition} {3rd}\ ed.\ (\bibinfo  {publisher}
  {Springer-Verlag},\ \bibinfo {year} {2004})\BibitemShut {NoStop}%
\bibitem [{\citenamefont {Evans}(2013)}]{evans2013stochastic}%
  \BibitemOpen
  \bibfield  {author} {\bibinfo {author} {\bibfnamefont {L.~C.}\ \bibnamefont
  {Evans}},\ }\href@noop {} {\emph {\bibinfo {title} {An Introduction to
  Stochastic Differential Equations}}},\ \bibinfo {edition} {1st}\ ed.\
  (\bibinfo  {publisher} {American Mathematical Society},\ \bibinfo {year}
  {2013})\BibitemShut {NoStop}%
\bibitem [{\citenamefont {{\O}ksendal}(2013)}]{oksendal2013stochastic}%
  \BibitemOpen
  \bibfield  {author} {\bibinfo {author} {\bibfnamefont {B.}~\bibnamefont
  {{\O}ksendal}},\ }\href@noop {} {\emph {\bibinfo {title} {Stochastic
  differential equations}}},\ \bibinfo {edition} {6th}\ ed.\ (\bibinfo
  {publisher} {Springer-Verlag},\ \bibinfo {year} {2013})\BibitemShut {NoStop}%
\bibitem [{\citenamefont {H{\'e}non}(1976)}]{henon1976two}%
  \BibitemOpen
  \bibfield  {author} {\bibinfo {author} {\bibfnamefont {M.}~\bibnamefont
  {H{\'e}non}},\ }\href@noop {} {\bibfield  {journal} {\bibinfo  {journal}
  {Communications in Mathematical Physics}\ }\textbf {\bibinfo {volume} {50}},\
  \bibinfo {pages} {69} (\bibinfo {year} {1976})}\BibitemShut {NoStop}%
\bibitem [{\citenamefont {Bastani}\ and\ \citenamefont
  {Hosseini}(2007)}]{Bastani2007}%
  \BibitemOpen
  \bibfield  {author} {\bibinfo {author} {\bibfnamefont {A.~F.}\ \bibnamefont
  {Bastani}}\ and\ \bibinfo {author} {\bibfnamefont {S.~M.}\ \bibnamefont
  {Hosseini}},\ }\href {\doibase 10.1016/j.cam.2006.08.012} {\bibfield
  {journal} {\bibinfo  {journal} {J Comp Appl Math}\ }\textbf {\bibinfo
  {volume} {206}},\ \bibinfo {pages} {631 } (\bibinfo {year}
  {2007})}\BibitemShut {NoStop}%
\bibitem [{\citenamefont {Burrage}(1999)}]{BurragePhD}%
  \BibitemOpen
  \bibfield  {author} {\bibinfo {author} {\bibfnamefont {P.~M.}\ \bibnamefont
  {Burrage}},\ }\emph {\bibinfo {title} {Runge-Kutta methods for stochastic
  differential equations}},\ \href@noop {} {Ph.D. thesis},\ \bibinfo  {school}
  {U Queensland, Australia} (\bibinfo {year} {1999})\BibitemShut {NoStop}%
\bibitem [{\citenamefont {Schwabedal}\ and\ \citenamefont
  {Pikovsky}(2013)}]{schwabedal2013phase}%
  \BibitemOpen
  \bibfield  {author} {\bibinfo {author} {\bibfnamefont {J.~T.~C.}\
  \bibnamefont {Schwabedal}}\ and\ \bibinfo {author} {\bibfnamefont
  {A.}~\bibnamefont {Pikovsky}},\ }\href@noop {} {\bibfield  {journal}
  {\bibinfo  {journal} {Physical review letters}\ }\textbf {\bibinfo {volume}
  {110}},\ \bibinfo {pages} {204102} (\bibinfo {year} {2013})}\BibitemShut
  {NoStop}%
\bibitem [{\citenamefont {Thomas}\ and\ \citenamefont
  {Lindner}(2014)}]{thomas2014asymptotic}%
  \BibitemOpen
  \bibfield  {author} {\bibinfo {author} {\bibfnamefont {P.~J.}\ \bibnamefont
  {Thomas}}\ and\ \bibinfo {author} {\bibfnamefont {B.}~\bibnamefont
  {Lindner}},\ }\href@noop {} {\bibfield  {journal} {\bibinfo  {journal}
  {Physical review letters}\ }\textbf {\bibinfo {volume} {113}},\ \bibinfo
  {pages} {254101} (\bibinfo {year} {2014})}\BibitemShut {NoStop}%
\bibitem [{\citenamefont {Cao}\ \emph {et~al.}(2020)\citenamefont {Cao},
  \citenamefont {Lindner},\ and\ \citenamefont {Thomas}}]{cao2020partial}%
  \BibitemOpen
  \bibfield  {author} {\bibinfo {author} {\bibfnamefont {A.}~\bibnamefont
  {Cao}}, \bibinfo {author} {\bibfnamefont {B.}~\bibnamefont {Lindner}}, \ and\
  \bibinfo {author} {\bibfnamefont {P.~J.}\ \bibnamefont {Thomas}},\ }\href
  {\doibase 10.1137/18M1218601} {\bibfield  {journal} {\bibinfo  {journal}
  {SIAM J. Appl. Math.}\ }\textbf {\bibinfo {volume} {80}},\ \bibinfo {pages}
  {422} (\bibinfo {year} {2020})}\BibitemShut {NoStop}%
\bibitem [{\citenamefont {Engel}\ and\ \citenamefont
  {Kuehn}(2021)}]{engel2021rdsphase}%
  \BibitemOpen
  \bibfield  {author} {\bibinfo {author} {\bibfnamefont {M.}~\bibnamefont
  {Engel}}\ and\ \bibinfo {author} {\bibfnamefont {C.}~\bibnamefont {Kuehn}},\
  }\href {\doibase 10.1007/s00220-021-04077-z} {\bibfield  {journal} {\bibinfo
  {journal} {Comm. Math. Phys.}\ }\textbf {\bibinfo {volume} {386}},\ \bibinfo
  {pages} {1603} (\bibinfo {year} {2021})}\BibitemShut {NoStop}%
\bibitem [{\citenamefont {Briers}\ \emph {et~al.}(2009)\citenamefont {Briers},
  \citenamefont {Doucet},\ and\ \citenamefont {Maskell}}]{Briers-2009-smo}%
  \BibitemOpen
  \bibfield  {author} {\bibinfo {author} {\bibfnamefont {M.}~\bibnamefont
  {Briers}}, \bibinfo {author} {\bibfnamefont {A.}~\bibnamefont {Doucet}}, \
  and\ \bibinfo {author} {\bibfnamefont {S.}~\bibnamefont {Maskell}},\ }\href
  {\doibase 10.1007/s10463-009-0236-2} {\bibfield  {journal} {\bibinfo
  {journal} {Annals of the Institute of Statistical Mathematics}\ }\textbf
  {\bibinfo {volume} {62}},\ \bibinfo {pages} {61} (\bibinfo {year}
  {2009})}\BibitemShut {NoStop}%
\bibitem [{\citenamefont {Hirsch}\ \emph {et~al.}(1977)\citenamefont {Hirsch},
  \citenamefont {Pugh},\ and\ \citenamefont {Shub}}]{hirsch-invariant-1977}%
  \BibitemOpen
  \bibfield  {author} {\bibinfo {author} {\bibfnamefont {M.~W.}\ \bibnamefont
  {Hirsch}}, \bibinfo {author} {\bibfnamefont {C.~C.}\ \bibnamefont {Pugh}}, \
  and\ \bibinfo {author} {\bibfnamefont {M.}~\bibnamefont {Shub}},\ }\href@noop
  {} {\emph {\bibinfo {title} {Invariant manifolds}}},\ Lecture Notes in
  Mathematics\ (\bibinfo  {publisher} {Springer-Verlag, Berlin-New York},\
  \bibinfo {year} {1977})\ pp.\ \bibinfo {pages} {ii+149}\BibitemShut {NoStop}%
\bibitem [{\citenamefont {Winfree}(1980)}]{dk.Winfree.book}%
  \BibitemOpen
  \bibfield  {author} {\bibinfo {author} {\bibfnamefont {A.~T.}\ \bibnamefont
  {Winfree}},\ }\href@noop {} {\emph {\bibinfo {title} {The Geometry of
  Biological Time}}}\ (\bibinfo  {publisher} {Springer-Verlag},\ \bibinfo
  {address} {New York},\ \bibinfo {year} {1980})\BibitemShut {NoStop}%
\bibitem [{\citenamefont {Kostelich}\ and\ \citenamefont
  {Yorke}(1990)}]{shKos90}%
  \BibitemOpen
  \bibfield  {author} {\bibinfo {author} {\bibfnamefont {E.~J.}\ \bibnamefont
  {Kostelich}}\ and\ \bibinfo {author} {\bibfnamefont {J.~A.}\ \bibnamefont
  {Yorke}},\ }\href {\doibase 10.1016/0167-2789(90)90121-5} {\bibfield
  {journal} {\bibinfo  {journal} {Phys D: Nonlinear Phenom}\ }\textbf {\bibinfo
  {volume} {41}},\ \bibinfo {pages} {183} (\bibinfo {year} {1990})}\BibitemShut
  {NoStop}%
\bibitem [{\citenamefont {Takens}(1981)}]{takens1981detecting}%
  \BibitemOpen
  \bibfield  {author} {\bibinfo {author} {\bibfnamefont {F.}~\bibnamefont
  {Takens}},\ }in\ \href@noop {} {\emph {\bibinfo {booktitle} {Dynamical
  systems and turbulence, {W}arwick 1980 ({C}oventry, 1979/1980)}}},\ \bibinfo
  {series} {Lecture Notes in Math.}, Vol.\ \bibinfo {volume} {898}\ (\bibinfo
  {publisher} {Springer, Berlin-New York},\ \bibinfo {year} {1981})\ pp.\
  \bibinfo {pages} {366--381}\BibitemShut {NoStop}%
\bibitem [{\citenamefont {Schelter}\ \emph {et~al.}(2006)\citenamefont
  {Schelter}, \citenamefont {Winterhalder},\ and\ \citenamefont
  {Timmer}}]{schelter2006tsa}%
  \BibitemOpen
  \bibinfo {editor} {\bibfnamefont {B.}~\bibnamefont {Schelter}}, \bibinfo
  {editor} {\bibfnamefont {M.}~\bibnamefont {Winterhalder}}, \ and\ \bibinfo
  {editor} {\bibfnamefont {J.}~\bibnamefont {Timmer}},\ eds.,\ \href {\doibase
  10.1002/9783527609970} {\emph {\bibinfo {title} {Handbook of time series
  analysis: recent theoretical developments and applications}}}\ (\bibinfo
  {publisher} {Wiley},\ \bibinfo {year} {2006})\ \bibinfo {note} {iSBN:
  9783527406234}\BibitemShut {NoStop}%
\bibitem [{\citenamefont {Gibson}\ \emph {et~al.}(1992)\citenamefont {Gibson},
  \citenamefont {Farmer}, \citenamefont {Casdagli},\ and\ \citenamefont
  {Eubank}}]{gibson1992}%
  \BibitemOpen
  \bibfield  {author} {\bibinfo {author} {\bibfnamefont {J.~F.}\ \bibnamefont
  {Gibson}}, \bibinfo {author} {\bibfnamefont {J.~D.}\ \bibnamefont {Farmer}},
  \bibinfo {author} {\bibfnamefont {M.}~\bibnamefont {Casdagli}}, \ and\
  \bibinfo {author} {\bibfnamefont {S.}~\bibnamefont {Eubank}},\ }\href
  {\doibase 10.1016/0167-2789(92)90085-2} {\bibfield  {journal} {\bibinfo
  {journal} {Physica D: Nonlinear Phenomena}\ }\textbf {\bibinfo {volume}
  {57}},\ \bibinfo {pages} {1 } (\bibinfo {year} {1992})}\BibitemShut {NoStop}%
\bibitem [{\citenamefont {Huguet}\ and\ \citenamefont {de-la
  Llave}(2012)}]{huguet2012}%
  \BibitemOpen
  \bibfield  {author} {\bibinfo {author} {\bibfnamefont {G.}~\bibnamefont
  {Huguet}}\ and\ \bibinfo {author} {\bibfnamefont {R.}~\bibnamefont {de-la
  Llave}},\ }\href {http://www.ima.umn.edu/preprints/dec2012/2415.pdf} {\emph
  {\bibinfo {title} {Computation of limit cycles and their isochrones: fast
  algorithms and their convergence}}},\ \bibinfo {type} {Tech. Rep.}\ \bibinfo
  {number} {2415}\ (\bibinfo  {institution} {Institute for mathematics and its
  applications, University of Minnesota},\ \bibinfo {year} {2012})\BibitemShut
  {NoStop}%
\bibitem [{\citenamefont {Osinga}\ and\ \citenamefont
  {Moehlis}(2010)}]{osinga2010isoch}%
  \BibitemOpen
  \bibfield  {author} {\bibinfo {author} {\bibfnamefont {H.~M.}\ \bibnamefont
  {Osinga}}\ and\ \bibinfo {author} {\bibfnamefont {J.}~\bibnamefont
  {Moehlis}},\ }\href {\doibase 10.1137/090777244} {\bibfield  {journal}
  {\bibinfo  {journal} {SIAM Journal on Applied Dynamical Systems}\ }\textbf
  {\bibinfo {volume} {9}},\ \bibinfo {pages} {1201} (\bibinfo {year}
  {2010})}\BibitemShut {NoStop}%
\bibitem [{\citenamefont {Revzen}\ and\ \citenamefont
  {Guckenheimer}(2012)}]{revzen2010fdsd}%
  \BibitemOpen
  \bibfield  {author} {\bibinfo {author} {\bibfnamefont {S.}~\bibnamefont
  {Revzen}}\ and\ \bibinfo {author} {\bibfnamefont {J.~M.}\ \bibnamefont
  {Guckenheimer}},\ }\href {\doibase 10.1098/rsif.2011.0431} {\bibfield
  {journal} {\bibinfo  {journal} {J R Soc Lond Interface}\ }\textbf {\bibinfo
  {volume} {9}},\ \bibinfo {pages} {957} (\bibinfo {year} {2012})}\BibitemShut
  {NoStop}%
\bibitem [{\citenamefont {Peters}\ and\ \citenamefont
  {Schaal}(2008)}]{peters2008}%
  \BibitemOpen
  \bibfield  {author} {\bibinfo {author} {\bibfnamefont {J.}~\bibnamefont
  {Peters}}\ and\ \bibinfo {author} {\bibfnamefont {S.}~\bibnamefont
  {Schaal}},\ }\href {\doibase 10.1016/j.neucom.2007.11.026} {\bibfield
  {journal} {\bibinfo  {journal} {Neurocomputing}\ }\textbf {\bibinfo {volume}
  {71}},\ \bibinfo {pages} {1180 } (\bibinfo {year} {2008})}\BibitemShut
  {NoStop}%
\bibitem [{\citenamefont {Schaal}\ and\ \citenamefont
  {Atkeson}(2010)}]{schall-atkeson-2010}%
  \BibitemOpen
  \bibfield  {author} {\bibinfo {author} {\bibfnamefont {S.}~\bibnamefont
  {Schaal}}\ and\ \bibinfo {author} {\bibfnamefont {C.}~\bibnamefont
  {Atkeson}},\ }\href {\doibase 10.1109/MRA.2010.936957} {\bibfield  {journal}
  {\bibinfo  {journal} {Robotics Automation Magazine, IEEE}\ }\textbf {\bibinfo
  {volume} {17}},\ \bibinfo {pages} {20 } (\bibinfo {year} {2010})}\BibitemShut
  {NoStop}%
\bibitem [{\citenamefont {Schaal}\ \emph {et~al.}(2005)\citenamefont {Schaal},
  \citenamefont {Peters}, \citenamefont {Nakanishi},\ and\ \citenamefont
  {Ijspeert}}]{schaal2005learning}%
  \BibitemOpen
  \bibfield  {author} {\bibinfo {author} {\bibfnamefont {S.}~\bibnamefont
  {Schaal}}, \bibinfo {author} {\bibfnamefont {J.}~\bibnamefont {Peters}},
  \bibinfo {author} {\bibfnamefont {J.}~\bibnamefont {Nakanishi}}, \ and\
  \bibinfo {author} {\bibfnamefont {A.}~\bibnamefont {Ijspeert}},\ }in\ \href
  {\doibase 10.1007/11008941\%5f60} {\emph {\bibinfo {booktitle} {Robotics
  Research}}}\ (\bibinfo  {publisher} {Springer},\ \bibinfo {year} {2005})\
  pp.\ \bibinfo {pages} {561--572}\BibitemShut {NoStop}%
\bibitem [{\citenamefont {Pikovsky}(2015)}]{pikovsky2015comment}%
  \BibitemOpen
  \bibfield  {author} {\bibinfo {author} {\bibfnamefont {A.}~\bibnamefont
  {Pikovsky}},\ }\href@noop {} {\bibfield  {journal} {\bibinfo  {journal}
  {Physical review letters}\ }\textbf {\bibinfo {volume} {115}},\ \bibinfo
  {pages} {069401} (\bibinfo {year} {2015})}\BibitemShut {NoStop}%
\bibitem [{\citenamefont {Thomas}\ and\ \citenamefont
  {Lindner}(2015)}]{thomas2015reply}%
  \BibitemOpen
  \bibfield  {author} {\bibinfo {author} {\bibfnamefont {P.~J.}\ \bibnamefont
  {Thomas}}\ and\ \bibinfo {author} {\bibfnamefont {B.}~\bibnamefont
  {Lindner}},\ }\href@noop {} {\bibfield  {journal} {\bibinfo  {journal}
  {Physical review letters}\ }\textbf {\bibinfo {volume} {115}},\ \bibinfo
  {pages} {069402} (\bibinfo {year} {2015})}\BibitemShut {NoStop}%
\bibitem [{\citenamefont {Mauroy}\ and\ \citenamefont
  {Mezi{\'c}}(2012)}]{mauroy2012use}%
  \BibitemOpen
  \bibfield  {author} {\bibinfo {author} {\bibfnamefont {A.}~\bibnamefont
  {Mauroy}}\ and\ \bibinfo {author} {\bibfnamefont {I.}~\bibnamefont
  {Mezi{\'c}}},\ }\href@noop {} {\bibfield  {journal} {\bibinfo  {journal}
  {Chaos: An Interdisciplinary Journal of Nonlinear Science}\ }\textbf
  {\bibinfo {volume} {22}},\ \bibinfo {pages} {033112} (\bibinfo {year}
  {2012})}\BibitemShut {NoStop}%
\bibitem [{\citenamefont {Kvalheim}\ and\ \citenamefont
  {Revzen}(2021)}]{kvalheim2021existence}%
  \BibitemOpen
  \bibfield  {author} {\bibinfo {author} {\bibfnamefont {M.~D.}\ \bibnamefont
  {Kvalheim}}\ and\ \bibinfo {author} {\bibfnamefont {S.}~\bibnamefont
  {Revzen}},\ }\href {\doibase 10.1016/j.physd.2021.132959} {\bibfield
  {journal} {\bibinfo  {journal} {Physica D}\ }\textbf {\bibinfo {volume}
  {425}},\ \bibinfo {pages} {Paper No. 132959, 20} (\bibinfo {year}
  {2021})}\BibitemShut {NoStop}%
\bibitem [{\citenamefont {Kvalheim}\ \emph
  {et~al.}(2021{\natexlab{a}})\citenamefont {Kvalheim}, \citenamefont {Hong},\
  and\ \citenamefont {Revzen}}]{kvalheim2021generic}%
  \BibitemOpen
  \bibfield  {author} {\bibinfo {author} {\bibfnamefont {M.~D.}\ \bibnamefont
  {Kvalheim}}, \bibinfo {author} {\bibfnamefont {D.}~\bibnamefont {Hong}}, \
  and\ \bibinfo {author} {\bibfnamefont {S.}~\bibnamefont {Revzen}},\
  }\href@noop {} {\bibfield  {journal} {\bibinfo  {journal}
  {IFAC-PapersOnLine}\ }\textbf {\bibinfo {volume} {54}},\ \bibinfo {pages}
  {267} (\bibinfo {year} {2021}{\natexlab{a}})}\BibitemShut {NoStop}%
\bibitem [{\citenamefont {Mauroy}\ \emph {et~al.}(2013)\citenamefont {Mauroy},
  \citenamefont {Mezi{\'c}},\ and\ \citenamefont
  {Moehlis}}]{mauroy2013isostables}%
  \BibitemOpen
  \bibfield  {author} {\bibinfo {author} {\bibfnamefont {A.}~\bibnamefont
  {Mauroy}}, \bibinfo {author} {\bibfnamefont {I.}~\bibnamefont {Mezi{\'c}}}, \
  and\ \bibinfo {author} {\bibfnamefont {J.}~\bibnamefont {Moehlis}},\
  }\href@noop {} {\bibfield  {journal} {\bibinfo  {journal} {Physica D:
  Nonlinear Phenomena}\ }\textbf {\bibinfo {volume} {261}},\ \bibinfo {pages}
  {19} (\bibinfo {year} {2013})}\BibitemShut {NoStop}%
\bibitem [{\citenamefont {Brunton}\ \emph {et~al.}(2021)\citenamefont
  {Brunton}, \citenamefont {Budi{\v{s}}i{\'c}}, \citenamefont {Kaiser},\ and\
  \citenamefont {Kutz}}]{brunton2021modern}%
  \BibitemOpen
  \bibfield  {author} {\bibinfo {author} {\bibfnamefont {S.~L.}\ \bibnamefont
  {Brunton}}, \bibinfo {author} {\bibfnamefont {M.}~\bibnamefont
  {Budi{\v{s}}i{\'c}}}, \bibinfo {author} {\bibfnamefont {E.}~\bibnamefont
  {Kaiser}}, \ and\ \bibinfo {author} {\bibfnamefont {J.~N.}\ \bibnamefont
  {Kutz}},\ }\href@noop {} {\bibfield  {journal} {\bibinfo  {journal} {arXiv
  preprint arXiv:2102.12086}\ } (\bibinfo {year} {2021})}\BibitemShut {NoStop}%
\bibitem [{\citenamefont {Kutz}\ \emph {et~al.}(2016)\citenamefont {Kutz},
  \citenamefont {Brunton}, \citenamefont {Brunton},\ and\ \citenamefont
  {Proctor}}]{kutz2016dynamic}%
  \BibitemOpen
  \bibfield  {author} {\bibinfo {author} {\bibfnamefont {J.~N.}\ \bibnamefont
  {Kutz}}, \bibinfo {author} {\bibfnamefont {S.~L.}\ \bibnamefont {Brunton}},
  \bibinfo {author} {\bibfnamefont {B.~W.}\ \bibnamefont {Brunton}}, \ and\
  \bibinfo {author} {\bibfnamefont {J.~L.}\ \bibnamefont {Proctor}},\
  }\href@noop {} {\emph {\bibinfo {title} {Dynamic mode decomposition:
  data-driven modeling of complex systems}}}\ (\bibinfo  {publisher}
  {S{I}{A}{M}},\ \bibinfo {year} {2016})\BibitemShut {NoStop}%
\bibitem [{\citenamefont {Wu}\ \emph {et~al.}(2021)\citenamefont {Wu},
  \citenamefont {Brunton},\ and\ \citenamefont {Revzen}}]{wu2021challenges}%
  \BibitemOpen
  \bibfield  {author} {\bibinfo {author} {\bibfnamefont {Z.}~\bibnamefont
  {Wu}}, \bibinfo {author} {\bibfnamefont {S.~L.}\ \bibnamefont {Brunton}}, \
  and\ \bibinfo {author} {\bibfnamefont {S.}~\bibnamefont {Revzen}},\
  }\href@noop {} {\bibfield  {journal} {\bibinfo  {journal} {Journal of the
  Royal Society Interface}\ }\textbf {\bibinfo {volume} {18}},\ \bibinfo
  {pages} {20210686} (\bibinfo {year} {2021})}\BibitemShut {NoStop}%
\bibitem [{\citenamefont {Revzen}\ \emph {et~al.}(2011)\citenamefont {Revzen},
  \citenamefont {Guckenheimer},\ and\ \citenamefont {Full}}]{revzen2011sicb}%
  \BibitemOpen
  \bibfield  {author} {\bibinfo {author} {\bibfnamefont {S.}~\bibnamefont
  {Revzen}}, \bibinfo {author} {\bibfnamefont {J.~M.}\ \bibnamefont
  {Guckenheimer}}, \ and\ \bibinfo {author} {\bibfnamefont {R.~J.}\
  \bibnamefont {Full}},\ }in\ \href@noop {} {\emph {\bibinfo {booktitle}
  {Yearly meeting of the Society for Integrative and Comparative Biology}}}\
  (\bibinfo {organization} {Society for Integrative and Comparative Biology},\
  \bibinfo {year} {2011})\BibitemShut {NoStop}%
\bibitem [{\citenamefont {Revzen}(2009)}]{RevzenPhD09}%
  \BibitemOpen
  \bibfield  {author} {\bibinfo {author} {\bibfnamefont {S.}~\bibnamefont
  {Revzen}},\ }\emph {\bibinfo {title} {Neuromechanical Control Architectures
  in Arthropod Locomotion}},\ \href@noop {} {Ph.D. thesis},\ \bibinfo  {school}
  {University of California, Berkeley} (\bibinfo {year} {2009}),\ \bibinfo
  {note} {department of Integrative Biology}\BibitemShut {NoStop}%
\bibitem [{\citenamefont {Tytell}(2013)}]{Tytell-SICB2013}%
  \BibitemOpen
  \bibfield  {author} {\bibinfo {author} {\bibfnamefont {E.~D.}\ \bibnamefont
  {Tytell}},\ }in\ \href@noop {} {\emph {\bibinfo {booktitle} {Soc Integ Compar
  Biol}}}\ (\bibinfo {year} {2013})\ p.\ \bibinfo {pages} {40.1}\BibitemShut
  {NoStop}%
\bibitem [{\citenamefont {Wang}\ and\ \citenamefont
  {Srinivasan}(2012)}]{Wang-Srinivasan-2012}%
  \BibitemOpen
  \bibfield  {author} {\bibinfo {author} {\bibfnamefont {Y.}~\bibnamefont
  {Wang}}\ and\ \bibinfo {author} {\bibfnamefont {M.}~\bibnamefont
  {Srinivasan}},\ }in\ \href {\doibase 10.1115/DSCC2012-MOVIC2012-8663} {\emph
  {\bibinfo {booktitle} {ASME 5th Annual Dynamic Systems and Control
  Conference}}},\ Vol.~\bibinfo {volume} {2}\ (\bibinfo {organization} {ASME},\
  \bibinfo {year} {2012})\ pp.\ \bibinfo {pages} {19--23}\BibitemShut {NoStop}%
\bibitem [{\citenamefont {Ankarali}\ and\ \citenamefont
  {Cowan}(2014)}]{Ankarali-2014-htf}%
  \BibitemOpen
  \bibfield  {author} {\bibinfo {author} {\bibfnamefont {M.~M.}\ \bibnamefont
  {Ankarali}}\ and\ \bibinfo {author} {\bibfnamefont {N.~J.}\ \bibnamefont
  {Cowan}},\ }in\ \href {\doibase 10.1109/CDC.2014.7039515} {\emph {\bibinfo
  {booktitle} {53rd IEEE Conference on Decision and Control}}}\ (\bibinfo
  {organization} {IEEE},\ \bibinfo {year} {2014})\ pp.\ \bibinfo {pages}
  {1017--1022}\BibitemShut {NoStop}%
\bibitem [{\citenamefont {Back}\ \emph {et~al.}(1993)\citenamefont {Back},
  \citenamefont {Guckenheimer},\ and\ \citenamefont
  {Myers}}]{Back_Guckenheimer_Myers_1993}%
  \BibitemOpen
  \bibfield  {author} {\bibinfo {author} {\bibfnamefont {A.}~\bibnamefont
  {Back}}, \bibinfo {author} {\bibfnamefont {J.}~\bibnamefont {Guckenheimer}},
  \ and\ \bibinfo {author} {\bibfnamefont {M.}~\bibnamefont {Myers}},\
  }\href@noop {} {\bibfield  {journal} {\bibinfo  {journal} {Hybrid Systems}\
  ,\ \bibinfo {pages} {255–267}} (\bibinfo {year} {1993})}\BibitemShut
  {NoStop}%
\bibitem [{\citenamefont {Simi{\'c}}\ \emph {et~al.}(2000)\citenamefont
  {Simi{\'c}}, \citenamefont {Johansson}, \citenamefont {Sastry},\ and\
  \citenamefont {Lygeros}}]{simic2000towards}%
  \BibitemOpen
  \bibfield  {author} {\bibinfo {author} {\bibfnamefont {S.~N.}\ \bibnamefont
  {Simi{\'c}}}, \bibinfo {author} {\bibfnamefont {K.~H.}\ \bibnamefont
  {Johansson}}, \bibinfo {author} {\bibfnamefont {S.}~\bibnamefont {Sastry}}, \
  and\ \bibinfo {author} {\bibfnamefont {J.}~\bibnamefont {Lygeros}},\ }in\
  \href@noop {} {\emph {\bibinfo {booktitle} {International Workshop on Hybrid
  Systems: Computation and Control}}}\ (\bibinfo {organization} {Springer},\
  \bibinfo {year} {2000})\ pp.\ \bibinfo {pages} {421--436}\BibitemShut
  {NoStop}%
\bibitem [{\citenamefont {Westervelt}\ \emph {et~al.}(2003)\citenamefont
  {Westervelt}, \citenamefont {Grizzle},\ and\ \citenamefont
  {Koditschek}}]{westervelt2003hybrid}%
  \BibitemOpen
  \bibfield  {author} {\bibinfo {author} {\bibfnamefont {E.~R.}\ \bibnamefont
  {Westervelt}}, \bibinfo {author} {\bibfnamefont {J.~W.}\ \bibnamefont
  {Grizzle}}, \ and\ \bibinfo {author} {\bibfnamefont {D.~E.}\ \bibnamefont
  {Koditschek}},\ }\href@noop {} {\bibfield  {journal} {\bibinfo  {journal}
  {IEEE transactions on automatic control}\ }\textbf {\bibinfo {volume} {48}},\
  \bibinfo {pages} {42} (\bibinfo {year} {2003})}\BibitemShut {NoStop}%
\bibitem [{\citenamefont {Ames}\ and\ \citenamefont
  {Sastry}(2005)}]{ames2005homology}%
  \BibitemOpen
  \bibfield  {author} {\bibinfo {author} {\bibfnamefont {A.~D.}\ \bibnamefont
  {Ames}}\ and\ \bibinfo {author} {\bibfnamefont {S.}~\bibnamefont {Sastry}},\
  }in\ \href@noop {} {\emph {\bibinfo {booktitle} {International Workshop on
  Hybrid Systems: Computation and Control}}}\ (\bibinfo {organization}
  {Springer},\ \bibinfo {year} {2005})\ pp.\ \bibinfo {pages}
  {86--102}\BibitemShut {NoStop}%
\bibitem [{\citenamefont {Goebel}\ \emph {et~al.}(2009)\citenamefont {Goebel},
  \citenamefont {Sanfelice},\ and\ \citenamefont
  {Teel}}]{Goebel_Sanfelice_Teel_2009}%
  \BibitemOpen
  \bibfield  {author} {\bibinfo {author} {\bibfnamefont {R.}~\bibnamefont
  {Goebel}}, \bibinfo {author} {\bibfnamefont {R.~G.}\ \bibnamefont
  {Sanfelice}}, \ and\ \bibinfo {author} {\bibfnamefont {A.}~\bibnamefont
  {Teel}},\ }\href@noop {} {\bibfield  {journal} {\bibinfo  {journal} {Control
  Systems, IEEE}\ }\textbf {\bibinfo {volume} {29}},\ \bibinfo {pages}
  {28–93} (\bibinfo {year} {2009})}\BibitemShut {NoStop}%
\bibitem [{\citenamefont {Lerman}(2016)}]{lerman2016category}%
  \BibitemOpen
  \bibfield  {author} {\bibinfo {author} {\bibfnamefont {E.}~\bibnamefont
  {Lerman}},\ }\href@noop {} {\bibfield  {journal} {\bibinfo  {journal} {arXiv
  preprint arXiv:1612.01950}\ } (\bibinfo {year} {2016})}\BibitemShut {NoStop}%
\bibitem [{\citenamefont {Lerman}\ and\ \citenamefont
  {Schmidt}(2020)}]{lerman2020networks}%
  \BibitemOpen
  \bibfield  {author} {\bibinfo {author} {\bibfnamefont {E.}~\bibnamefont
  {Lerman}}\ and\ \bibinfo {author} {\bibfnamefont {J.}~\bibnamefont
  {Schmidt}},\ }\href {\doibase 10.1016/j.geomphys.2019.103582} {\bibfield
  {journal} {\bibinfo  {journal} {J. Geom. Phys.}\ }\textbf {\bibinfo {volume}
  {149}},\ \bibinfo {pages} {103582, 30} (\bibinfo {year} {2020})}\BibitemShut
  {NoStop}%
\bibitem [{\citenamefont {Clark}\ and\ \citenamefont
  {Bloch}(2021)}]{clark2021invariant}%
  \BibitemOpen
  \bibfield  {author} {\bibinfo {author} {\bibfnamefont {W.}~\bibnamefont
  {Clark}}\ and\ \bibinfo {author} {\bibfnamefont {A.}~\bibnamefont {Bloch}},\
  }\href@noop {} {\bibfield  {journal} {\bibinfo  {journal} {arXiv preprint
  arXiv:2101.11128}\ } (\bibinfo {year} {2021})}\BibitemShut {NoStop}%
\bibitem [{\citenamefont {Kvalheim}\ \emph
  {et~al.}(2021{\natexlab{b}})\citenamefont {Kvalheim}, \citenamefont
  {Gustafson},\ and\ \citenamefont {Koditschek}}]{kvalheim2021conley}%
  \BibitemOpen
  \bibfield  {author} {\bibinfo {author} {\bibfnamefont {M.~D.}\ \bibnamefont
  {Kvalheim}}, \bibinfo {author} {\bibfnamefont {P.}~\bibnamefont {Gustafson}},
  \ and\ \bibinfo {author} {\bibfnamefont {D.~E.}\ \bibnamefont {Koditschek}},\
  }\href {\doibase 10.1137/20M1336576} {\bibfield  {journal} {\bibinfo
  {journal} {SIAM J. Appl. Dyn. Syst.}\ }\textbf {\bibinfo {volume} {20}},\
  \bibinfo {pages} {784} (\bibinfo {year} {2021}{\natexlab{b}})}\BibitemShut
  {NoStop}%
\bibitem [{Note6()}]{Note6}%
  \BibitemOpen
  \bibinfo {note} {In more detail: fix $x\in \protect \mathbf {X}$, let
  $U\subsetneq \Gamma $ be a contractible open neighborhood of $\protect
  \mathrm {P}(x)$, and define $\tau _x\in [0,\infty )$ to be the smallest
  nonnegative number such that $\protect \mathrm {P}\circ \Phi _{-\tau
  _x}({x})=o(0)$. Since the isochron $\protect \mathrm {P}^{-1}(o(0))$ is a
  $\protect \mathcal {C}^{k\geq 1}$ manifold transverse to the vector field
  \cite {shGuc75, eldering2018global}, the implicit function theorem implies
  that there is a $\protect \mathcal {C}^k$ function $\tau \protect \nobreak
  \mskip 2mu\mathpunct {}\nonscript \mkern -\thinmuskip {:}\mskip
  6muplus1mu\relax \protect \mathrm {P}^{-1}(U)\to \protect \mathbb {R}$
  satisfying $\tau (x)=\tau _x$ and $\protect \mathrm {P}\circ \Phi _{-\tau
  (y)}({y})=o(0)$ for all $y\in \protect \mathrm {P}^{-1}(U)$. Thus, $\protect
  \mathrm {P}|_{\protect \mathrm {P}^{-1}(U)} = o\circ \tau $ is $\protect
  \mathcal {C}^k$. Since $x\in \protect \mathbf {X}$ was arbitrary, it follows
  that $\protect \mathrm {P}\in \protect \mathcal {C}^k$.}\BibitemShut {Stop}%
\bibitem [{Note7()}]{Note7}%
  \BibitemOpen
  \bibinfo {note} {The geometry of isochrons can be surprisingly complicated,
  especially when the vector field $f$ has multiple time scales \protect \citep
  {langfield2015forward}}\BibitemShut {NoStop}%
\bibitem [{Note8()}]{Note8}%
  \BibitemOpen
  \bibinfo {note} {In a Euclidean space, $\alpha $ can be expressed as $\alpha
  :=(f^\protect \mathsf {T}f)^{-1}\omega f^\protect \mathsf {T}$}\BibitemShut
  {NoStop}%
\bibitem [{\citenamefont {Farber}\ \emph {et~al.}(2004)\citenamefont {Farber},
  \citenamefont {Kappeler}, \citenamefont {Latschev},\ and\ \citenamefont
  {Zehnder}}]{farber2004lyapunov}%
  \BibitemOpen
  \bibfield  {author} {\bibinfo {author} {\bibfnamefont {M.}~\bibnamefont
  {Farber}}, \bibinfo {author} {\bibfnamefont {T.}~\bibnamefont {Kappeler}},
  \bibinfo {author} {\bibfnamefont {J.}~\bibnamefont {Latschev}}, \ and\
  \bibinfo {author} {\bibfnamefont {E.}~\bibnamefont {Zehnder}},\ }\href
  {\doibase 10.1017/S0143385703000762} {\bibfield  {journal} {\bibinfo
  {journal} {Ergodic theory and dynamical systems}\ }\textbf {\bibinfo {volume}
  {24}},\ \bibinfo {pages} {1451} (\bibinfo {year} {2004})}\BibitemShut
  {NoStop}%
\bibitem [{Note9()}]{Note9}%
  \BibitemOpen
  \bibinfo {note} {In fact, a \protect \textit {strong} deformation retract. As
  an embedded submanifold of $U$, the limit cycle is a strong deformation
  retract of a ``{tubular neighborhood}'' (\cite {lee2013smooth} Chapter 6)
  containing it. Using the flow, one may construct a strong deformation retract
  of $U$ onto a tubular neighborhood. Following this homotopy with the strong
  deformation retraction of the tubular neighborhood onto the limit cycle then
  gives the desired strong deformation retraction.}\BibitemShut {Stop}%
\bibitem [{\citenamefont {Pajitnov}(2006)}]{pajitnov2006circle}%
  \BibitemOpen
  \bibfield  {author} {\bibinfo {author} {\bibfnamefont {A.~V.}\ \bibnamefont
  {Pajitnov}},\ }\href {\doibase 10.1515/9783110197976} {\emph {\bibinfo
  {title} {Circle-valued {M}orse theory}}},\ \bibinfo {series} {De Gruyter
  Studies in Mathematics}, Vol.~\bibinfo {volume} {32}\ (\bibinfo  {publisher}
  {Walter de Gruyter \& Co., Berlin},\ \bibinfo {year} {2006})\ pp.\ \bibinfo
  {pages} {x+454}\BibitemShut {NoStop}%
\bibitem [{\citenamefont {Kvalheim}(2018)}]{KvalheimPhD18}%
  \BibitemOpen
  \bibfield  {author} {\bibinfo {author} {\bibfnamefont {M.}~\bibnamefont
  {Kvalheim}},\ }\emph {\bibinfo {title} {Aspects of invariant manifold theory
  and applications}},\ \href@noop {} {Ph.D. thesis},\ \bibinfo  {school}
  {University of Michigan} (\bibinfo {year} {2018}),\ \bibinfo {note}
  {department of Electrical Engineering and Computer Science}\BibitemShut
  {NoStop}%
\bibitem [{\citenamefont {Doob}(1990)}]{doob1953stochastic}%
  \BibitemOpen
  \bibfield  {author} {\bibinfo {author} {\bibfnamefont {J.~L.}\ \bibnamefont
  {Doob}},\ }\href@noop {} {\emph {\bibinfo {title} {Stochastic processes}}},\
  Wiley Classics Library\ (\bibinfo  {publisher} {John Wiley \& Sons, Inc., New
  York},\ \bibinfo {year} {1990})\ pp.\ \bibinfo {pages} {viii+654},\ \bibinfo
  {note} {reprint of the 1953 original, A Wiley-Interscience
  Publication}\BibitemShut {NoStop}%
\bibitem [{\citenamefont {Andrews}(1987)}]{andrews1987consistency}%
  \BibitemOpen
  \bibfield  {author} {\bibinfo {author} {\bibfnamefont {D.~W.~K.}\
  \bibnamefont {Andrews}},\ }\href {\doibase 10.2307/1913568} {\bibfield
  {journal} {\bibinfo  {journal} {Econometrica}\ }\textbf {\bibinfo {volume}
  {55}},\ \bibinfo {pages} {1465} (\bibinfo {year} {1987})}\BibitemShut
  {NoStop}%
\bibitem [{\citenamefont {P\"{o}tscher}\ and\ \citenamefont
  {Prucha}(1989)}]{potscher1989uniform}%
  \BibitemOpen
  \bibfield  {author} {\bibinfo {author} {\bibfnamefont {B.~M.}\ \bibnamefont
  {P\"{o}tscher}}\ and\ \bibinfo {author} {\bibfnamefont {I.~R.}\ \bibnamefont
  {Prucha}},\ }\href {\doibase 10.2307/1911058} {\bibfield  {journal} {\bibinfo
   {journal} {Econometrica}\ }\textbf {\bibinfo {volume} {57}},\ \bibinfo
  {pages} {675} (\bibinfo {year} {1989})}\BibitemShut {NoStop}%
\bibitem [{Note10()}]{Note10}%
  \BibitemOpen
  \bibinfo {note} {Note that, by the Arzel\`a-Ascoli theorem, the compact
  subsets of $\Omega ^1_0(K)$ are precisely the closed and bounded subsets of
  $1$-forms which are equicontinuous.}\BibitemShut {Stop}%
\bibitem [{\citenamefont {Aliprantis}\ and\ \citenamefont
  {Border}(2006)}]{aliprantis2006infinite}%
  \BibitemOpen
  \bibfield  {author} {\bibinfo {author} {\bibfnamefont {C.~D.}\ \bibnamefont
  {Aliprantis}}\ and\ \bibinfo {author} {\bibfnamefont {K.~C.}\ \bibnamefont
  {Border}},\ }\href@noop {} {\emph {\bibinfo {title} {Infinite dimensional
  analysis}}},\ \bibinfo {edition} {3rd}\ ed.\ (\bibinfo  {publisher}
  {Springer, Berlin},\ \bibinfo {year} {2006})\ pp.\ \bibinfo {pages}
  {xxii+703},\ \bibinfo {note} {a hitchhiker's guide}\BibitemShut {NoStop}%
\end{thebibliography}%

\clearpage
\appendix
\newcommand{\Heading}[1]{\subsection{#1}}
\newcommand{\subHeading}[1]{\subsubsection{#1}}
\section*{Supplemental Information (SI)}\label{sec:appendix}

\section{Algorithm notes}
The Supplemental Information is written with the mathematically inclined reader in mind. As such, we presume familiarity with fundamental concepts of differential topology \citep{guillemin1974differential,hirsch1976differential,lee2013smooth}, differential forms \citep{spivak1971calculus,guillemin1974differential,lee2013smooth}, and stochastic differential equations \cite{arnold1974stochastic,gardiner2004stochastic,evans2013stochastic,oksendal2013stochastic}.
\Heading{Generating ground truth test data---details}
\label{sec:groundtruthdata}

This section provides details of the method we used to generate randomized test systems with adjustible complexity and noise characteristics.
We provide these details in the hope that other investigators may use them to evaluate the performance of algorithms which aim to analyze the structure of nonlinear oscillators.

Very few oscillators admit a closed-form representation of their phase as a function of state, making validation of our algorithm a significant challenge all on its own.
To produce a well specified class of tests for our algorithm, we produced sample paths arising from a Stratonovich stochastic differential equation of the form
\begin{align}
\Der x(t) = f(x(t)) \Der t + G(x(t)) \circ \Der W(t)\label{eqn:genSDE},
\end{align}
and requiring that the deterministic dynamics of $\Der x(t) = f(x(t)) \Der t$ are known to be a (hyperbolic) limit cycle oscillator, with a known phase coordinate.

We obtained such a deterministic oscillator by starting with a randomly chosen linearized (Floquet) structure converging on the trajectory $x(t) = [t,0,\ldots]$, then winding this affine system around the unit circle with $x_1$ as the angle, and finally pushing forward the dynamics through a composition of randomly generated diffeomorphisms to generate $f(\cdot)$ of \refEqn{genSDE}.
We chose diffeomorphisms such that their tangent maps were computable and invertible in closed form.
We used a second composition of random diffeomorphisms to generate a function $g(\cdot)$, and the inverse of the tangent map of $g(\cdot)$ was used for $G(\cdot)$.
By using the inverse diffeomorphisms, we directly computed the phase of each trajectory point with respect to the deterministic dynamics, to use as a baseline for comparing with other forms of phase estimation.

\subHeading{Implementation of the Floquet Structure}\label{sec:floq}
While our exposition so far focused on differential forms, allowing us to emphasize the coordinate-invariant properties of the \ToF, we provide the equations in this section in a computationally friendly coordinate dependent matrix notation.

Given an $n$-dimensional state space ($n>1$), we defined a vector field by first taking the polar decomposition $\Omega:\, \Real^n \rightarrow S^1 \times \Real^+ \times \Real^{n-2}$ of the first two coordinates: $p \defeq [x_3, x_4,\ldots]$, $\theta \defeq\mathrm{atan2}(x_1,x_2)$, $r\defeq \sqrt{x_1^2+x_2^2}$ and thus  $\Omega(x) = [\theta, r] \oplus p$.

In these polar coordinates, the equation of motion we implemented is affine:
\begin{align}
\dot \theta &= 1 \nonumber\\
\dot r &= \beta (r-1) +  c_r^\T p \\
\dot p &= \gamma (r-1) + M p \label{eqn:floq} \nonumber
\end{align}

This means that $\dot \theta$ is, for a deterministic system, uniformly advancing in time, and the latter two equations comprise an autonomous linear subsystem which evolves independently of $\theta$.
These equations can be combined to form a single equation in homogeneous coordinates:

\begin{align}
[\dot \theta, \dot r, \dot p, 0 ]^\T = \tilde M [ \theta, r, p, 1 ]^\T
\end{align}
where the matrix $\tilde M$ can be chosen to obtain any eigenvalue structure desired for the dynamics.

\subHeading{A natural class of invertible diffeomorphisms}\label{sec:henon}

To produce randomized invertible diffeomorphisms, we used a structure inspired by the H\'{e}non map \citep{henon1976two} (and therefore refer to these as H-maps) and recommended to us by J.M. Guckenheimer.
Given: A (split) vector space $\Spc{Q} \defeq \Spc{X} \oplus \Spc{Y}$; Invertible maps $g_X:\, \Spc{X} \rightarrow \Spc{X}$ and $g_Y:\, \Spc{Y} \rightarrow \Spc{Y}$; and a (possibly non-invertible) map  $f:\, \Spc{Y} \rightarrow \Spc{X}$, we constructed the following mapping:

\begin{align}
&\tilde{x} \defeq g_X(x) + (f \circ g_Y)(y) &
&\tilde{y} \defeq g_Y(y),
\end{align}

which admits the inverse (by construction):

\begin{align}
& x \defeq g^{-1}_X\left( \tilde{x} - f (\tilde{y}) \right) &
& y \defeq g^{-1}_Y(\tilde{y}).
\end{align}

In our implementation we have chosen to make $g_X$ and $g_Y$ affine maps.
For our nonlinear map $f(\cdot)$, we used inversions of the form

\begin{align}
f(x) & \defeq m\paren{\beta+2}\frac{\varrho\paren{x}P\paren{x}}{\paren{\varrho\paren{x}^2 + \beta\varrho\paren{x} + 1}} \\ \varrho(x) & \defeq \paren{x^\T A x},
\end{align}

for some positive symmetric $A$ with eigenvalues close to one (uniformly randomly sampled between $0.95$ and $1.05$ for our simulations).
$f$ is a function which is approximately zero when $\varrho\paren{x}$ is small and large, and is stationary with value $m$ at $\varrho\paren{x}=1$.
It has one stationary point for positive values of $\varrho\paren{x}$ with $\beta$ controlling how ``peaked'' the function is (the more negative beta is, the greater the ``peaking'').
$P$ maps from $\Spc{X}$ of dimension $D_X$ to $\Spc{Y}$ of dimension $D_Y$ and the components of $P$ are given by.

\begin{equation}
P_{ij} = \left\{ \begin{array}{ll}
1 & \mbox{if $i \equiv j \,[\mod D_X]$} \\
0 & \mbox{if $i \not\equiv j\,[\mod D_X]$}
\end{array} \right.
\end{equation}

with $1 \leq i \leq D_Y$ and $1 \leq j \leq D_X$.

We composed together a sequence of these maps to produce a highly nonlinear, but invertible, distortion of the Floquet system after it has been transformed to polar coordinates.
Between each pair of consecutive H-maps, we inserted a randomly chosen full-rank affine transformation.

\subHeading{Stochastic integration of the combined system}

We implemented individual diffeomorphisms as objects of a \texttt{Python} class using the {\sffamily SciPy} scientific programming environment.
Each mapping object was capable of forward and inverse mappings, pushforwards and pullbacks.
We composed these mapping objects with each other to construct the fully formed random diffeomorphisms for both the deterministic and noise terms of \refEqn{genSDE}.
The noise term uses the Jacobian of an H-map chain to transform the Wiener process when integrating the SDE.
The final integrated Floquet system is then transformed by a different H-map chain.
We integrated the SDE of \refEqn{genSDE} using an implementation of the R3 stochastic integration scheme \citep{Bastani2007, BurragePhD}, and have made this integrator available publicly at \url{http://github/BIRDSlab/BIRDSode}.
We also note that while the ``ground truth'' isochrons of the deterministic system are known precisely, they may not be compatible with various ``ground truth'' definitions of stochastic isochrons \citep{schwabedal2013phase,thomas2014asymptotic,cao2020partial,engel2021rdsphase} for the SDE.
For example, the presence of noise can change the average frequency of a limit cycle oscillator, and by extension, may deform the isochrons.
We expect that while at the low noise limit our recovered phase will closely match the deterministic system (this expectation is mathematically justified  in \refSec{tof_from_uncertain_core_result}), as noise is increased the discrepancy between the two may {\em justifiably} increase.
We therefore caution the reader that in the presence of consistent systematic errors and higher noise conditions, errors in estimates of asymptotic phase may represent an error in the ``ground truth'' rather than an estimation error on the part of the algorithm.

\subHeading{Preparation of data for phase estimation}

In each simulation, we treated these $D$ dimensional data similarly to how experimental data would be treated.
We filtered the data using a Kalman smoother \cite{Briers-2009-smo} with the system states the position and its derivative.
The state transition matrix assumed no acceleration and that the last $D$ coordinates were derivatives of the first $D$ coordinates.
The observation matrix was a $D\times 2D$ matrix with an identity in the first $D\times D$ sub-matrix.
We assumed the system noise matrix to be diagonal with the first $D$ system variables having the same level of variance, and the last $D$ variables having a (typically) different shared level variance.
We estimated these two system covariance parameters by maximizing the likelihood of the training observations.

After filtering, we performed a principal component analysis and rotated the system into the corresponding orthogonal coordinate system.
We then z-scored the first two principal components.

\Heading{Review of differential forms}\label{sec:dforms}

Differential forms $\Omega$ over $\Xs$ are an exterior algebra generated using two operators: the \concept{exterior product} $\wedge$---the universal skew symmetric product, and the \concept{exterior derivative} $\Deriv$.
Elements of rank 0 correspond to (sufficiently smooth) scalar valued functions $\Xs\to\Complex$.
When $\alpha\in\Omega_r$ and $\beta\in\Omega_s$, $\Deriv \alpha \in \Omega_{r+1}$ and $\alpha\wedge\beta \in \Omega_{r+s}$.

These operations are familiar to many in their 3-dimensional special cases, where ranks $\Omega_0 \ldots \Omega_3$ correspond to scalar functions, vectors, directed areas, and directed volumes.
Here the exterior derivative $\Deriv$ functions as the gradient, curl, and divergence.
The exterior product $\wedge$ functions as a scalar-vector product when one argument is rank 0, as a cross product (resulting in a directed area) when the arguments are rank 1, and as a box-product when applied to rank 2 and rank 1 arguments.

Formally defined, $\Cont^k$ differential 1-forms $\alpha$ on $\Xs$ are $\Cont^k$ sections of its cotangent bundle $\T ^*\Xs$: $\alpha(x) \in \T ^*_x \Xs$.
If $v$ is a $\Cont^k$ vector field, then $\langle \alpha, v \rangle$ is a $\Cont^k$ scalar function.
If $\alpha$ is a $\Cont^1$ curve in $\Xs$, the line integral $\int_\alpha \alpha$ is well defined and independent of coordinate system.

A 1-form $\alpha$ is \concept{closed} if its exterior derivative vanishes.
It is \concept{exact} if it is the exterior derivative of a scalar function.
The identity $\Deriv \circ \Deriv = 0$ implies that all exact forms are closed.
On contractible regions of $\Xs$ the converse is true, i.e. all closed forms are exact (a.k.a. the Poincar\'e Lemma).
The quotient space of closed to exact $k$-forms is known as the $k$-th de-Rham cohomology group.
We use some properties of this cohomology to construct our algorithm.

If a closed curve $\gamma_1$ can be smoothly deformed into a curve $\gamma_2$, $\int_{\gamma_1} \alpha$ of a closed 1-form $\alpha$ equals $\int_{\gamma_2} \alpha$.
In other words, the integral of a closed 1-form depends only on the homotopy class of the underlying closed curve.

A more complete introduction to the theory of differential forms can be found in \cite{spivak1971calculus,guillemin1974differential,lee2013smooth}.

\subsection{\ToF}
\label{sec:mathbg}
This self-contained section serves as an alternative introduction to our paper which contains more details and is aimed at mathematically inclined readers.
It includes a description of the \ToF\ assuming some familiarity with differential forms, which are reviewed in \refSec{dforms}.

We consider continuous-time dynamical systems, or flows $\FlwSym$, on a smooth $n$-dimensional manifold $\Xs$.
That $\FlwSym$ is a flow means that $\FlwSym:\Xs \times \Real \to \Xs$, $\FlwSym_{0} = \textnormal{id}_\Xs$, and $\FlwSym_{t}\circ \FlwSym_{s}=\FlwSym_{t+s}$.
We assume that the vector field $f \defeq \frac{\partial}{\partial t}\FlwSym_{t}|_{t = 0}$ is well-defined and $\Cont^{k\geq 1}$, so that $\Phi\in \Cont^k$.
The trajectories $t\mapsto \Flw{t}{x_0}\eqqcolon x(t)$ satisfy the ordinary differential equation (ODE)
 $\dot x = f(x)$.
A periodic orbit $o$ of period $T>0$ is a trajectory that satisfies $o(t+T) = o(t)$ with $o(t) \ne o(0)$ for all $0< t < T$.
We will always assume that $T>0$ is the \emph{minimal} period of a nonstationary periodic orbit.
We focus on dynamical systems possessing only one nonstationary periodic trajectory $o$ and denote by $\cycle\subset \Xs$ the image of this $o$.

We further assume that this periodic trajectory is exponentially stable, a property that holds in many practical cases.
Exponentially stable limit cycles are always \concept{normally hyperbolic} \citep{hirsch-invariant-1977, eldering2018global}; for brevity we will simply refer to these as \concept{limit cycles}.
We refer to the stability basin of a limit cycle and the dynamics within it as an \concept{oscillator}.
The asymptotic behavior of any oscillator is described fully by its \concept{asymptotic phase}:
the stability basin of an oscillator is partitioned into codimension-$1$ $\Cont^k$ embedded submanifolds \cite[pp.~4208--4209]{eldering2018global} called \concept{isochrons} which are asymptotically in phase with one another \cite{shGuc75,dk.Winfree.book}.

For systems in which the dimension is low, the noise is small or the data is plentiful, data pairs $(x_k,{\dot x}_k),~k=1\ldots N$ may be used to estimate the vector field $f$.
Numerous investigators proposed such approaches, primarily in the 1990s.
Classical papers include Kostelich and Yorke \cite{shKos90}, who fit arbitrary continuous dynamics to trajectories, and the rich literature on time-delay embeddings, derivatives and principal components \cite{takens1981detecting,sauer1991embedology, schelter2006tsa, gibson1992}.
When dynamical equations are known precisely, automatic differentiation and continuation methods can be used to compute the isochrons, a problem equivalent to that of phase estimation \cite{huguet2012, osinga2010isoch}.
However, care must be taken if it is desired to extract phase response curves from this, as these quantities are coordinate dependent \citep{wilshin2021phase}.

With fewer data points and higher levels of noise, one expects that the full off-cycle dynamics are less amenable to accurate estimation, and that only a few of the slower-decaying ``modes'' might be observable off of the limit cycle \cite{revzen2010fdsd}.
At the extreme, only phase itself might be detectable, through methods such as we proposed in Revzen \& Guckenheimer \cite{RevGuk08} using various phase-locking approaches \cite{peters2008, schall-atkeson-2010, schaal2005learning}.
This motivates the main goal of the present paper: estimate asymptotic phase from empirically observed dynamics, e.g., from an ensemble of noisy measurements of (possibly short and noisy) system trajectory segments.

We note that, while there has been work to define generalized notions of asymptotic phase for stochastic oscillators \citep{schwabedal2013phase,thomas2014asymptotic,cao2020partial,engel2021rdsphase} (see \cite{pikovsky2015comment,thomas2015reply} for a spirited discussion), in this paper we restrict ourselves to estimation of the classical asymptotic phase of a deterministic oscillator using data from a perturbed version of the underlying deterministic system.
We also note that there are operator-theoretic methods of recent interest revealing phase as the generator of a family of eigenfunctions of the Koopman operator with purely imaginary eigenvalues \cite{mauroy2012use,kvalheim2021existence,kvalheim2021generic}.
These eigenfunctions and others can be estimated using Fourier/Laplace averages \cite{mauroy2012use,mauroy2013isostables, kvalheim2021existence, kvalheim2021generic, brunton2021modern} when dynamical equations are known, and from data using Dynamic Mode Decomposition \cite{kutz2016dynamic}, at least in certain cases.
However, the results of these emerging spectral methods seem to be sensitive to the choice of observables and their nonlinearities \citep{wu2021challenges}.

Several investigators have pursued the construction of linearized (Floquet) models for dynamics near the limit cycle.
We proposed an approach termed \concept{Data Driven Floquet Analysis (DDFA)} \cite{revzen2011sicb, RevzenPhD09} consisting of estimating phase, then constructing affine models conditioned on phase.
Tytell \cite{Tytell-SICB2013} reported a DDFA approach based on harmonic balance.
Wang and Srinivasan \cite{Wang-Srinivasan-2012} proposed to construct a Floquet model using ``factored Poincar\'e maps'' and derive a phase estimate from this model.
Ankarali \textit{et. al} \cite{Ankarali-2014-htf} applied \concept{harmonic transfer functions} and achieved good results on a \concept{hybrid} \cite{Back_Guckenheimer_Myers_1993,simic2000towards,westervelt2003hybrid, ames2005homology, Goebel_Sanfelice_Teel_2009,lerman2016category,lerman2020networks, clark2021invariant, kvalheim2021conley} spring-mass hopper model.
Regardless of how they are obtained, linear models of the dynamics around the limit cycle can at best only represent the hyperplanes tangent to the isochrons at their intersection with the limit cycle.
By contrast, the method we propose here can construct nonlinear isochron approximations on suitable neighborhoods of the limit cycle, as we explain in \refSec{tof_from_uncertain}.

\newcommand{\pphi}{\varphi}
\newcommand{\pphic}{\pphi|_{_{\cycle}}}
We now begin our theoretical description of the object we call the \concept{\ToF}.
By redefining $\Xs$ to be the basin of attraction of $\cycle$, we henceforth assume that $\cycle$ is globally asymptotically stable.
Asymptotic phase can be viewed as a map $\AP:\Xs \to \cycle$ which is a retraction ($\AP|_{\cycle} = \textnormal{id}|_{\cycle}$) and a semiconjugacy  ($\forall t\in\Real:\AP \circ \FlwSym_{t} = \FlwSym_{t}\circ \AP$).
From this it follows that $\AP\in \Cont^k$ when $f\in \Cont^k$\footnote{
  In more detail: fix $x\in \Xs$, let $U\subsetneq \cycle$ be a contractible open neighborhood of $\AP(x)$, and define $\tau_x\in [0,\infty)$ to be the smallest nonnegative number such that $\AP\circ \Flw{-\tau_x}{x}=o(0)$.    Since the isochron $\AP^{-1}(o(0))$ is a $\Cont^{k\geq 1}$ manifold transverse to the vector field \cite{shGuc75, eldering2018global},  the implicit function theorem implies that there is a $\Cont^k$ function $\tau\colon \AP^{-1}(U)\to \R$ satisfying $\tau(x)=\tau_x$ and $\AP \circ \Flw{-\tau(y)}{y}=o(0)$ for all $y\in \AP^{-1}(U)$.
  Thus, $\AP|_{\AP^{-1}(U)} = o\circ \tau$ is $\Cont^k$.
  Since $x\in \Xs$ was arbitrary, it follows that $\AP\in \Cont^k$.
}.
Each $p \in \cycle$ represents the \concept{isochron} $\AP^{-1}(p)$, which is the set of initial conditions from which trajectories converge with the one starting at $p$\footnote{
  The geometry of isochrons can be surprisingly complicated, especially when the vector field $f$ has multiple time scales \citep{langfield2015forward}
}.
In the present context of oscillators, asymptotic phase is more commonly represented as a circle-valued (``phasor''-valued) map $\pphi\colon \Xs\to \Sph^1$; we explain the relationships between different representations of phase in \refSec{equiv}.

The \ToF\ appears naturally as a consequence of the existence of $\AP$.
Noting that $f$ is nowhere zero on the $1$-dimensional $\Cont^k$ manifold $\cycle$ and defining the angular frequency $\omega \coloneqq 2\pi/T$, it follows that there exists a unique $\Cont^{k-1}$ differential 1-form $\alpha$ on $\cycle$ satisfying $\langle \alpha, f \rangle = \omega$ identically on $\cycle$\footnote{In a Euclidean space, $\alpha$ can be expressed as $\alpha\defeq (f^\T f)^{-1}\omega  f^\T$}.
The \concept{\ToF\ } $\dT:\,\Xs \to \T^*\Xs$ is defined to be the pullback $\AP^*\alpha$ of $\alpha$, which means that $\langle\dT(x),v\rangle \defeq \langle\alpha(\AP(x)),\D\AP(x)\cdot v\rangle$ for any $v\in \T_x \Xs$.
Since $\AP\in \Cont^k$, it follows that $\dT\in \Cont^{k-1}$.
It is easy to see that the \ToF\ satisfies two properties (further discussed in \refSec{equiv}):  (1) $\langle \dT, f \rangle = \omega$ everywhere, and (2) $\dT$ is closed. 

By \concept{closed} we mean that each $x\in \Xs$ possesses an open neighborhood $U_x$ on which $\dT$ is the exterior derivative of a $\Cont^{k}$ function.
This definition permits us to discuss \emph{continuous} closed forms, needed for the case $f\in \Cont^1$; it reduces to the usual definition for $\Cont^1$ forms (cf. \cite{farber2004lyapunov}).

Somewhat less obvious is the fact that the \ToF\ is the \emph{unique} continuous closed $1$-form on $\Xs$ satisfying the preceding two properties; we show that this is the case in \refSec{equiv}.
In that section we also relate $\AP$, $\dT$, and circle-valued phases $\pphi\colon \Xs\to \Sph^1$.
By a \concept{circle-valued (asymptotic) phase} $\pphi\colon \Xs\to \Sph^1\subset \mathbb{C}$ we mean a continuous map satisfying $\pphi\circ \FlwSym_{t}=e^{i\omega t}\pphi$ for all $t\in \R$.
To summarize Theorem~\ref{th:phase-reps-relate} of \refSec{equiv}, some of these relationships are as follows: $\pphi \circ \AP = \pphi$, $\pphi$ is unique modulo rotations of $\Sph^1$, and $\dT = \pphi^*(d\theta)$ is the pullback of the standard angular form on $\Sph^1$ via any circle-valued phase.
Moreover, any choice of basepoint $x_0\in \Xs$ uniquely determines a circle-valued phase by integrating $\dT$ along continuous curves from $x_0$.

While we are interested in asymptotic phase as defined for deterministic oscillators, our primary goal is to compute asymptotic phase from real-world data which we assume may be subject to system and measurement noise.
In \refSec{tof_from_uncertain} we derive fairly general performance guarantees applicable to our algorithm and potentially others.
In that section, we define ``unobserved'' and ``observed'' empirical cost functions $J_N$ and $\eJ$ mapping $1$-forms to nonnegative numbers, where $N$ is the number of observed state-velocity data pairs.
Intuitively, we would like to minimize the former cost (for which $J_N(\dT)=0$), but only the latter cost is observable, so our algorithms attempts to minimize the latter.

Assuming that our algorithm outputs estimates $\dTe_N$ satisfying $\dTe_N \leq \dT + \epsilon$ for some $\epsilon \geq 0$, that $\Xs$ is an open subset of $\R^n$, and that state measurements belong to some fixed compact set $K\subset \Xs$, we establish in Theorem~\ref{th:cost_funcs_close} the following inequality holding with probability $1$: 
\begin{align*}
\limsup_{N\to\infty}& \frac{\sqrt{J_N(\dTe_N)}-2\sqrt{\epsilon}}{\|\dT\|_{\Cont^0(K)}+(1+\sqrt{2})\|\dTe_N\|_{\Cont^0(K)}} \stackrel{\textnormal{a.s.}}{\leq}2\sqrt{\kappa} .
\end{align*}
Here $\|\slot\|_{\Cont^0(K)}$ denotes the supremum norm.
Assuming $\epsilon = 0$ for simplicity, it follows in particular that, with $f$ fixed, the performance of the estimator with respect to $J_N$ is within $\bo(\kappa)$ of optimality in the limit of large data.
Moreover, the full statement of Theorem~\ref{th:cost_funcs_close} also includes explicit convergence rates (see Remark~\ref{rem:convergence-rate}).
See the remarks following Theorem~\ref{th:cost_funcs_close} for further discussion.

However, as discussed in \refSec{tof_from_uncertain}, bounding the performance with respect to $J_N$ as in Theorem~\ref{th:cost_funcs_close} does not say anything about closeness of the estimates $\dTe_N$ themselves to $\dT$.
Thus, even if $J_N(\dTe_N)$ is small, the estimated isochrons corresponding to $\dTe_N$ need not resemble the true isochrons.
This motivates Theorem~\ref{th:estimated_isochrons_approach_true_isochrons} which shows, roughly speaking, that under some additional assumptions the estimated isochrons converge $\Cont^1$-uniformly to the true isochrons on forward invariant open subsets of $K$.
Moreover, the full statement of Theorem~\ref{th:phase-reps-relate} also includes explicit convergence rates.


\Heading{Relationship of the \ToF\ to other representations of asymptotic phase}\label{sec:equiv}
In this section we explain the relationship (mentioned in \refSec{mathbg}) between the asymptotic phase map $\AP\colon \Xs\to \cycle$, circle-valued phase maps $\pphi\colon \Xs\to \Sph^1$, and the \ToF\ $\dT$.
We continue under the assumptions and notation of \refSec{mathbg}; in particular, the state space $\Xs$ is also the basin of $\cycle$.

\begin{lemma}\label{lem:t1f-properties}
  Consider an oscillator generated by the $\Cont^{1}$ vector field $f\colon \Xs\to \T\Xs$.
  The \ToF\ $\dT:\,\Xs\to\T^*\Xs$ is closed and satisfies $\langle \dT, f \rangle \equiv \omega$.
\end{lemma}
\begin{proof}
  Since $\AP\colon \Xs\to \cycle$ is a semiconjugacy, it follows that $\D\AP \circ f = f \circ \AP$.
  Since $\dT \coloneqq \AP^*\alpha$ (\refSec{mathbg}), it follows that $\langle \dT, f\rangle \equiv \langle P^*\alpha, f\rangle \equiv \langle \alpha, \D P \circ f\rangle \equiv \langle \alpha, f \circ P\rangle \equiv \omega$.
  Since $\alpha$ is closed and pullbacks commute with $\dform{}$, $\dT = \AP^*\alpha$ is closed.
\end{proof}

Given an open neighborhood $U\subset \Xs$ of $\cycle$ which is forward invariant for the flow of $f$, we extend the definition of circle-valued phase (\refSec{mathbg}) to include continuous maps  $\pphi\colon U\to \Sph^1\subset \mathbb{C}$ satisfying $\pphi\circ \FlwSym_{t} = e^{i\omega t}\pphi$ for all $t\geq 0$.
(We use this generality in the proof of Theorem~\ref{th:estimated_isochrons_approach_true_isochrons}.)
For the following lemmas and theorem, $d\theta$ denotes the standard angular form on the circle $\Sph^1$ and $i=\sqrt{-1}$.

\begin{lemma}\label{lem:unique}
 Consider an oscillator generated by the $\Cont^{1}$ vector field $f\colon \Xs\to \T\Xs$.
 Let $U$ be a forward invariant open neighborhood of $\cycle$.
 Given any $x_0\in U$, there exists a unique circle-valued phase $\pphi\colon U\to \Sph^1$ satisfying $\pphi(x_0)=1$.
 Moreover, $\pphi\in \Cont^{k\geq 1}$ if $f\in \Cont^k$.
\end{lemma}
\begin{proof}
Existence follows by first defining $\pphi|_\cycle\colon \cycle \to \Sph^1$ by $\pphi|_\cycle \circ \AP(x_0) =1$ and $\pphi|_\cycle \circ \FlwSym_{t} = e^{i\omega t}\pphi|_\cycle$, then defining $\pphi\coloneqq \pphi|_\cycle \circ \AP$.
(As explained in \refSec{mathbg}, $\AP\in \Cont^k$ when $f\in \Cont^k$.)
To show uniqueness, let $\pphi_1$ and $\pphi_2$ be two such circle-valued phases.
The ratio $\pphi_1/\pphi_2$ is constant along trajectories of $f$, so $\pphi_1/\pphi_2$ is constant on $\cycle$.
Since all trajectories converge to $\cycle$, it follows from continuity that $\pphi_1/\pphi_2\equiv \pphi_1(x_0)/\pphi_2(x_0)= 1$ on $U$, so $\pphi_1=\pphi_2$ as desired.
\end{proof}

\begin{lemma}\label{lem:t1f-to-circle-phase}
Consider an oscillator generated by the $\Cont^{1}$ vector field $f\colon \Xs\to \T\Xs$.
Let $U$ be a forward invariant open neighborhood of $\cycle$, and let $\beta$ be any $\Cont^{k\geq 0}$ closed $1$-form on $U$ satisfying $\langle \beta, f\rangle \equiv \omega$.
Then for any $x_0\in U$, there is a unique $\Cont^{k+1}$ circle-valued phase $\pphi\colon U\to \Sph^1$ satisfying $\pphi(x_0)=1$ and $\pphi^*(d\theta)=\beta$. 
\end{lemma}
\begin{proof}
  Since $U$ is forward invariant, it follows that $\cycle$ is a deformation retraction of $U$\footnote{In fact, a \textit{strong} deformation retract.
    As an embedded submanifold of $U$, the limit cycle is a strong deformation retract of a \concept{tubular neighborhood} (\cite{lee2013smooth} Chapter 6) containing it.
    Using the flow, one may construct a strong deformation retract of $U$ onto a tubular neighborhood.
    Following this homotopy with the strong deformation retraction of the tubular neighborhood onto the limit cycle then gives the desired strong deformation retraction.
  }. 
  Since $\beta$ is closed and satisfies $\langle \beta, f\rangle \equiv \omega$, it follows that the line integral of $\beta$ along any continuous closed curve in $U$ is an integer multiple of $2\pi$.
  Thus, exponentiating $i$ times the line integrals of $\beta$ along arbitrary continuous paths from $x_0$ defines a $\Cont^{k+1}$ map $\pphi\colon U\to \Sph^1$ satisfying  $\beta= \pphi^*(d\theta)$
  (cf. the proof of \cite[Lem.~1.17]{pajitnov2006circle}). 
  Moreover, the condition $\langle \beta, f\rangle \equiv \omega$ implies that $\pphi$ is a circle-valued phase. 
  Uniqueness of $\pphi$ follows from Lemma~\ref{lem:unique}.
\end{proof}

The following result summarizes some relationships between different representations of asymptotic phase.

\begin{theorem}\label{th:phase-reps-relate}
  Consider an oscillator generated by the $\Cont^{1}$ vector field $f\colon \Xs\to \T\Xs$, with asymptotic phase map denoted by $\AP\colon \Xs\to \cycle$.
  \begin{enumerate}
  \item\label{item:t1f-th-1}   The \ToF\ $\dT$ is the unique continuous closed $1$-form on $\Xs$ satisfying $$\langle \dT, f\rangle \equiv \omega.$$
  \item\label{item:t1f-th-2}  If $\pphi$ is any circle-valued phase, $\pphi\circ \AP = \pphi$ and  $$\dT = \pphi^*(d\theta) = -i(\dform \pphi)/\pphi.$$  
  \item\label{item:t1f-th-3} 
    Conversely, a choice of $x_0\in \Xs$ uniquely determines a circle-valued phase $\pphi$ satisfying $\pphi(x_0) = 1$ through the formula
      \begin{equation}
    \varphi(x) = \exp\left(i\int_{\gamma} \dT \right),
    \end{equation}
    where $\gamma$ is any $\Cont^0$ path joining $x_0$ to $x$.
  \end{enumerate}
\end{theorem}
\begin{proof}
Claim~\ref{item:t1f-th-3} is immediate from Lemma~\ref{lem:t1f-to-circle-phase} and its proof (taking $U=\Xs$).
The first statement of Claim~\ref{item:t1f-th-2} follows since $\pphi\circ \AP$ and $\pphi$ are both circle-valued phases agreeing on $\cycle$ (since $\AP|_\cycle = \textnormal{id}|_{\cycle}$), so they coincide by Lemma~\ref{lem:unique}.

Since any two circle-valued phases differ by a rotation of $\Sph^1$ (Lemma~\ref{lem:unique}), and since $d\theta$ is rotation invariant, Lemma~\ref{lem:t1f-to-circle-phase} implies Claim~\ref{item:t1f-th-1} and the equality $\dT = \pphi^*(d\theta)$ in Claim~\ref{item:t1f-th-2}.
The remaining equality in Claim~\ref{item:t1f-th-2} follows from the straightforward verification that $\pphi^*(d\theta) = -i(\dform \pphi)/\pphi$ for any $\Cont^1$ circle-valued map $\pphi$, where $\dform\pphi$ is viewed as a complex-valued differential $1$-form \citep[p.~192]{guillemin1974differential}.

\end{proof}

\Heading{The \ToF\ estimated from uncertain systems}
\label{sec:tof_from_uncertain}
In this section we prove two theorems relating the performance of the true \ToF\ to that of estimates computed in the presence of noise, to which data from real-world systems and measurements are subjected.
These are extensions of corresponding results in \citet[Sec.~3.4]{KvalheimPhD18}.

Throughout the remainder of this section, we assume for simplicity that the state space $\Xs$ (also assumed to be the basin of attraction of the limit cycle) is an open subset of $\Real^n$.
We also assume for the remainder of this section that data takes values in a fixed compact set $K\subset \Xs$.

In what follows we let $\|\slot\|$ be the Euclidean norm, and
we denote induced operator norms by the same symbol.
Given a continuous function $\beta$ on a set $A\subset \Xs$, we define
$$\|\beta\|_{\Cont^0(A)}\coloneqq \sup_{x\in A}\|\beta(x)\|.$$
We denote by $\Omega^1_0(A)$ the set of continuous $1$-forms over $A$, i.e., continuous sections of $\T^* \Xs|_A$.

\subHeading{Assumptions about Data and Noise}
\label{sec:assump-data-noise}
We assume that we have a finite collection of pairs $(x_{i},\dot x_{i})_{i=1}^N \subset K \times \R^n$, with $\dot x_{i}$ of the form $\dot x_{i} = f(x_{i}) + \eta_{i}$.
We consider the $x_i\in K$ and $\eta_i\in \R^n$ to be random variables, and we assume that there is a constant $\kappa \geq 0$ such that
\begin{equation}\label{eqn:eta_assumptions}
\forall i\colon \E(\|\eta_i\|^2) \leq \kappa.
\end{equation}
We also assume that the strong law of large numbers applies to yield the following equality with probability $1$; quite general conditions ensuring this can be found in \cite{doob1953stochastic, andrews1987consistency,potscher1989uniform}.
\begin{align}\label{eqn:ms-limsup}
\lim_{N\to\infty}\frac{1}{N}\sum_{i=1}^N\|\eta_i\|^2 - \E(\|\eta_i\|^2) &\stackrel{\textnormal{a.s.}}{=} 0
\end{align}
We will impose additional conditions involving the $x_i$ for Theorem~\ref{th:estimated_isochrons_approach_true_isochrons}, but not for Theorem~\ref{th:cost_funcs_close}.
\begin{Rem}\label{rem:noise-assump-generality}
  This setup is sufficiently general that the $\eta_i$  could arise from measurement noise, or system noise, or both.
  In \cite[Sec.~3.4.2]{KvalheimPhD18} it is argued in detail that this formulation applies to data from It\^{o} SDEs.
  It also applies to data coming from Stratonovich SDEs, but there is a slight twist: to get sharper results, one should replace $f$ by a modified vector field related to It\^{o}'s formula \cite[Sec.~3.4.2.2]{KvalheimPhD18}.
\end{Rem}

\subHeading{Performance of the estimated Temporal 1-Form}
\label{sec:tof_from_uncertain_core_result}

Denote by $\dT$ the unique true \ToF\ on $\Xs$ corresponding to the vector field $f\in \Cont^1$.
We define the ``unobserved'' cost function $J_N:\Omega_0^1(K)\to [0,\infty)$ which has $\dT\in \Cont^0$ as its (nonunique) global minimizer, with minimum $J_N(\dT)=0$, by
\begin{align}
J_N(\beta) &:= \frac{1}{N} \sum_{i=1}^{N} \left(\left\langle \beta(x_{i}), f(x_{i})\right\rangle - \omega\right)^2.
\end{align}
We also define the following ``observed'' cost function $\eJ:\Omega_0^1(K)\to[0,\infty)$ by
\begin{align}
\label{eqn:def_cost_J_N_noisy}
\eJ(\beta) &:= \frac{1}{N} \sum_{i=1}^{N} \left(\left\langle \beta(x_{i}), f(x_{i})+\eta_{i}\right\rangle - \omega\right)^2.
\end{align}
Fix $\epsilon \geq 0$.
For each $N$ we fix \emph{any} $\dTe_N \in \Omega_0^1(K)$ with the property that
\begin{equation}\label{eqn:assumption-phi-hat}
\eJ(\dTe_N) \leq \eJ(\dT) + \epsilon.
\end{equation}
We think of $\dTe_N$ as being the output of our algorithm to compute the \ToF; our algorithm effectively attempts to minimize $\eJ$.
We have the following performance bound on our algorithm (assuming its output satisfies \eqref{eqn:assumption-phi-hat}) with respect to minimizing $J_N$.

\newcommand{\Sqrt}[1]{\left[{#1}\right]^{_{{1}\over {2}}}}

\begin{theorem}\label{th:cost_funcs_close}Consider an oscillator generated by the $\Cont^{1}$ vector field $f\colon \Xs\to \T\Xs$.
Then
  \begin{align*}
 & \frac{1}{4}\left(\frac{\sqrt{J_N(\dTe_N)}-2\sqrt{\epsilon}}{\|\dT\|_{\Cont^0(K)}+(1+\sqrt{2})\|\dTe_N\|_{\Cont^0(K)}}\right)^2  \leq \frac{1}{N}\sum_{i=1}^N\|\eta_i\|^2.
  \end{align*}
Thus, under the assumptions of the present section,
  \begin{align*}
 \limsup_{N\to\infty}& \frac{\sqrt{J_N(\dTe_N)}-2\sqrt{\epsilon}}{\|\dT\|_{\Cont^0(K)}+(1+\sqrt{2})\|\dTe_N\|_{\Cont^0(K)}} \stackrel{\textnormal{a.s.}}{\leq}2\sqrt{\kappa} .
  \end{align*}
In particular, if $\|\dT\|_{\Cont^0(K)}, \|\dTe_N\|_{\Cont^0(K)} \leq  M_0$, then
  \begin{align*}
   \limsup_{N\to\infty}\sqrt{J_N(\dTe_N)} \stackrel{\textnormal{a.s.}}{\leq} (4+2\sqrt{2})M_0 \sqrt{\kappa} + 2\sqrt{\epsilon}.
  \end{align*}
\end{theorem}
We make several remarks before giving the proof.
\begin{Rem}
The assumptions \eqref{eqn:eta_assumptions} and \eqref{eqn:ms-limsup} are not needed for the first inequality of the theorem.
\end{Rem}

\begin{Rem}[Convergence rates]\label{rem:convergence-rate}
The first inequality in the theorem shows that the left side converges to $[0,\kappa]$ at least as fast as the rate of convergence in the law of large numbers \eqref{eqn:ms-limsup}.
Similarly, with respect to the \emph{weak} law of large numbers, the same inequality implies that the rate of convergence in probability of the left side are at least as fast as the corresponding rates of convergence in probability for the empirical second moments of the $\eta_i$.
Similar remarks hold for \eqref{eqn:form-conv-rate} in Theorem~\ref{th:estimated_isochrons_approach_true_isochrons}.
\end{Rem}

\begin{Rem}
The appearance of $\|\dT\|_{\Cont^0(K)}$ in the first and second bounds suggests that a (``wilder'') \ToF\ with large supremum norm is harder to estimate.
See also \eqref{eqn:form-conv-rate} in Theorem~\ref{th:estimated_isochrons_approach_true_isochrons}.
\end{Rem}
\begin{Rem}
Since $\dT\in \Cont^0$, there exists $M_0 > 0$ such that the assumption $\|\dT\|_{\Cont^0(K)}\leq M_0$ holds.
Since in practice the estimates $\dTe_N$ produced by our algorithm are uniformly bounded, for practical purposes it is also the case that $\|\dTe\|_{\Cont^0(K)}\leq M_0$ for some $M_0$ and all $N$.
Thus, the extra assumptions in the second portion of the theorem are quite mild.
\end{Rem}
\begin{Rem}
Since $J_N(\dT) = 0$, the estimate provided by Theorem \ref{th:cost_funcs_close} bounds how well our algorithm does at minimizing the ``unobserved'' cost function.
Taking $\epsilon = 0$ for simplicity, the final inequality of the theorem asserts that, with $f$ fixed, the performance of the estimator---as measured by $J_N$---is asymptotically within $\bo(\kappa)$ of optimality.
However, Theorem~\ref{th:cost_funcs_close} does \emph{not} say anything about how close the estimate $\dTe_N$ itself is to the true \ToF\ $\dT$.
This will be addressed in Theorem~\ref{th:estimated_isochrons_approach_true_isochrons}.
\end{Rem}
\begin{proof}[Proof of Theorem~\ref{th:cost_funcs_close}]
Expanding the summand in the definition of $\eJ$ and rearranging, we see that, for any $\beta \in \Omega^1_0(K)$: 
\begin{align}
\label{eq:Jhat_equality}
	J_N(\beta) &= \eJ(\beta) - \frac{1}{N}\sum_{i=1}^N\left\langle \beta(x_i),\eta_i \right \rangle^2  \nonumber\\
	& - \frac{2}{N}\sum_{i=1}^N \left(\left\langle \beta(x_i), f(x_i)\right\rangle - 1\right)\left \langle \beta(x_i), \eta_i \right \rangle\\
	\label{eq:Jhat_inequality}
	&\leq \eJ(\beta) + 2\Sqrt{J_N(\beta)}\Sqrt{\frac{1}{N}\sum_{i=1}^N\left\langle \beta(x_i),\eta_i \right \rangle^2} \nonumber\\
	& + \frac{1}{N}\sum_{i=1}^N\left\langle \beta(x_i),\eta_i \right \rangle^2, 
	\end{align}
	where the inequality follows from the Cauchy-Schwarz inequality.
	Since the true \ToF\ $\dT$ satisfies $(\langle\dT,f\rangle - 1)\equiv 0$, it follows immediately from \eqref{eq:Jhat_equality} that
	\begin{align}\label{eq:hat_J_true_dphi_equality}
	\eJ(\dT) &= \frac{1}{N}\sum_{i=1}^N \left \langle \dT(x_i),\eta_i \right \rangle^2.
	\end{align}
	Next,  \refEqn{assumption-phi-hat} and \eqref{eq:Jhat_inequality} imply that
	
	\begin{equation}\label{eq:J_N_ineq}
	\begin{split} 
	&J_N(\dTe_N) \leq \epsilon + \eJ(\dT) +  \frac{1}{N}\sum_{i=1}^N\left\langle \dTe_N(x_i),\eta_i \right \rangle^2 \\  
	&+ 2\Sqrt{J_N(\dTe_N)} \Sqrt{\frac{1}{N}\sum_{i=1}^N\left\langle \dTe_N(x_i),\eta_i \right \rangle^2}. 
	\end{split}
	\end{equation}
	Defining $S_N, T_N \geq 0$ by $$S_N^2\coloneqq \frac{1}{N}\sum_{i=1}^N \|\eta_i\|^2 , \quad T_N^2\coloneqq J_N(\dTe_N),$$
	substituting \eqref{eq:hat_J_true_dphi_equality} into \eqref{eq:J_N_ineq}, and using Cauchy-Schwarz yields an inequality quadratic in $T_N$:
	
	\begin{equation*}\label{eq:quadratic_ineq}
	\begin{split}
	T_N^2&-2S_N\|\dTe_N\|_{\Cont^0(K)}T_N- (\|\dT\|_{\Cont^0(K)}^2 + \|\dTe_N\|_{\Cont^0(K)}^2)S_N^2 \\
	&\leq \epsilon. 
	 \end{split}
	\end{equation*}
	This and the quadratic formula imply that 
	\begin{equation*}
	\begin{split}
	T_N&\leq 2S_N\|\dTe_N\|_{\Cont^0(K)}\\ &+ 2\Sqrt{(\|\dT\|_{\Cont^0(K)}^2 + 2\|\dTe_N\|_{\Cont^0(K)}^2)S_N^2 + \epsilon }.
	\end{split}
	\end{equation*}	
	Using subadditivity of $\sqrt{\slot}$, squaring, and rearranging yields
	$$\frac{T_N-2\sqrt{\epsilon}}{\|\dT\|_{\Cont^0(K)}+(1+\sqrt{2})\|\dTe_N\|_{\Cont^0(K)}}\leq 2S_N.$$
    This yields the first inequality of the theorem statement; taking the $\limsup$ of both sides as $N\to \infty$ and using \eqref{eqn:eta_assumptions} and \eqref{eqn:ms-limsup} completes the proof.
\end{proof}
We now begin preparations for the statement of Theorem~\ref{th:estimated_isochrons_approach_true_isochrons}.
Fix a compact subset $\mathcal{E} \subset \Omega^1_0(K)$ \footnote{Note that, by the Arzel\`a-Ascoli theorem, the compact subsets of $\Omega^1_0(K)$ are precisely the closed and bounded subsets of $1$-forms which are equicontinuous.} and define
\begin{equation*}
\F := \{\beta \in \mathcal{E}| \beta \textnormal{ is closed} \}.
\end{equation*}
It follows that $\F$ is compact since it is closed in $\Omega^1_0(K)$. 
This follows since the limit in $\Omega^1_0(K)$ of a sequence $(\beta_i) \subset \F$ must have integral zero over any nullhomotopic continuous loop in $K$, since the same is true of each $\beta_i$, so the limit is a closed $1$-form.

For Theorem~\ref{th:estimated_isochrons_approach_true_isochrons} we will assume that the strong \concept{uniform law of large numbers} applies to yield
\begin{equation}
\begin{split}\label{eqn:ulln}
\lim_{N\to\infty}\sup_{\beta\in \F}&\Bigg|J_N(\beta)
-\int_{K}\left(\langle \beta,f \rangle-\omega\right)^2\rho d\mu \Bigg|
\\&\stackrel{\textnormal{a.s.}}{=}0,
\end{split}
\end{equation}
where $\mu$ is some Borel probability measure on $K$.
Quite general conditions ensuring this can be found in \cite{andrews1987consistency,potscher1989uniform}.
Recall that the \concept{support} of a Borel measure $\mu$ is the set of all points $x$ such that every neighborhood of $x$ has positive measure.

\newcommand{\aff}{\textnormal{aff}}
We first need to prove the following lemma, in which $\aff(\slot)$ denotes the \concept{affine hull} of a subset of $\Omega^1_0(K)$.
\begin{lemma}\label{lem:coercivity}
Consider an oscillator generated by the $\Cont^1$ vector field $f\colon \Xs\to \T\Xs$.
Assume that $U\subset K$ is an open neighborhood of $\cycle$ in $\Xs$ which is forward invariant and contained in the support of $\mu$.
Also assume that $\dT \in \F$, that $\F$ is star-shaped with respect to $\dT$, and that there is $\delta > 0$ such that
\begin{equation}\label{eqn:relint}
\forall \beta \in \aff(\F)\colon \|\beta-\dT\|_{\Cont^0(U)} \leq \delta \implies \beta \in \F.
\end{equation}
Then there is $c > 0$ such that, for all $\beta\in -\dT + \F$, 
\begin{equation}\label{eqn:coercivity}
c \|\beta\|_{\Cont^0(U)}^2 \leq \int_U \langle \beta, f\rangle ^2 d\mu.
\end{equation}
\end{lemma}
\newcommand{\relb}{\textnormal{rel}\partial}
\newcommand{\tF}{\tilde{\F}}
\begin{Rem}\label{rem:convex}
A stronger assumption implying the star-shaped hypothesis is that $\F$ is convex, since then $\F$ is star-shaped with respect to all points in $\F$.
\end{Rem}
\begin{proof}
Define $Q\colon \Omega^1_0(K)\to [0,\infty)$ to send $\beta$ to the right side of \eqref{eqn:coercivity}, $\tF\coloneqq -\dT + \F$, and
\begin{equation}\label{eqn:B-def}
B\coloneqq \{\beta\in \aff(\tF)\colon \|\beta\|_{\Cont^0(U)} = \delta\}.
\end{equation}
Since $\F$ is compact and star-shaped with respect to $\dT$, it follows that $\tF$ is compact and star-shaped with respect to $0$.
Since $Q$ is continuous and $B$ is compact (by \eqref{eqn:relint}), $Q|_B$ attains a minimum $c_0\geq 0$. 

We claim that $c_0 > 0$.
Indeed, $Q(\beta)=0$ implies that $\langle \beta|_U, f|_U\rangle \equiv 0$, which implies that $\beta|_U = dh$ for some $\Cont^1$ function $h\colon U\to \R$ constant along the flow of $f$.
But $h$ must be constant everywhere since all trajectories converge to $\cycle$, so  $\|\beta\|_{\Cont^0(U)}=0$.
This and \eqref{eqn:B-def} imply that $\beta \not \in B$, so $c_0 > 0$.

Since $\tF$ is star-shaped with respect to $0$, each nonzero $\beta \in \tF$ is of the form $t \gamma$ for some $\gamma \in B$ and $t \coloneqq \|\beta\|_{\Cont^0(U)}/\delta$. 
Since $Q(t\gamma) = t^2 Q(\gamma) \geq t^2 c_0$, it follows that $Q(\beta) \geq \|\beta\|_{\Cont^0(U)}^2c_0/\delta^2$ for all $\beta \in \tF$.
Defining $c\coloneqq c_0/\delta^2$ completes the proof.
\end{proof}

\begin{theorem}\label{th:estimated_isochrons_approach_true_isochrons}
Consider an oscillator generated by the $\Cont^1$ vector field $f\colon \Xs\to \T\Xs$.
Assume that $U\subset K$ is an open neighborhood of $\cycle$ in $\Xs$ which is forward invariant and contained in the support of $\mu$.
Also assume that $\dTe_N, \dT \in \F$.
Then, under the assumptions of the present section, 
\begin{equation}\label{eqn:est-t1f-close-to-true}
 \lim_{\kappa + \epsilon \to 0}\limsup_{N\to\infty} \|\dTe_N - \dT\|_{\Cont^0(U)} \stackrel{\textnormal{a.s.}}{=} 0.
  \end{equation}
If also $\F$ is star-shaped with respect to $\dT$, \eqref{eqn:relint} holds for some $\delta > 0$, and $\zeta, \xi\colon \N\to [0,\infty)$ satisfy
\begin{equation}\label{eqn:conv-rates}
\begin{split}
&\frac{1}{N}\sum_{i=1}^N\|\eta_i\|^2 \leq \zeta(N)\\
&\Bigg|J_N(\dTe_N)
-\int_{K}\left(\langle \dTe_N,f \rangle-\omega\right)^2\rho d\mu \Bigg| \leq \xi(N),
\end{split}
\end{equation}
then the following explicit convergence rates hold:
\begin{equation}\label{eqn:form-conv-rate}
\begin{split}
&c\|\dTe_N - \dT\|_{\Cont^0(U)}^2 \\
&\leq 4\left[(\|\dT\|_{\Cont^0(K)}+(1+\sqrt{2})\|\dTe_N\|_{\Cont^0(K)})\zeta(N) + \sqrt{\epsilon}\right]^2\\
& + \xi(N)\\
&\leq 4 \left[(2+\sqrt{2})M_0\zeta(N) + \sqrt{\epsilon}\right]^2 + \xi(N),
\end{split}
\end{equation}
where $c>0$ is the constant in Lemma~\ref{lem:coercivity}, and the final inequality holds if $\|\dT\|_{\Cont^0(K)}, \|\dTe_N\|_{\Cont^0(K)} \leq  M_0$.
\end{theorem}
We make several remarks before giving the proof.
\begin{Rem}
As noted in Remark~\ref{rem:convex}, the star-shaped hypothesis is automatic if $\F$ is convex.
\end{Rem}
\begin{Rem}
The assumption \eqref{eqn:ulln} is not needed for the second portion of the theorem, since its role is replaced by \eqref{eqn:conv-rates}.
\end{Rem}
\begin{Rem}
Unlike Theorem~\ref{th:cost_funcs_close}, Theorem~\ref{th:estimated_isochrons_approach_true_isochrons} gives conditions under which the estimates $\dTe_N$ converge to the true \ToF\ $\dT$ uniformly on suitable neighborhoods $U\subset K$ of $\cycle$.
This convergence implies convergence of the corresponding isochron estimates to the true isochrons $\Cont^{1}$-uniformly on such neighborhoods.
This implies, roughly speaking, that the estimated isochrons corresponding to $\dTe_N$ will resemble the true isochrons near $\cycle$ up to one order of smoothness.
\end{Rem}
\begin{Rem}
For non-forward invariant $U$ the conclusions of the theorem do not generally hold, because it is possible to construct examples of such a $U$ on which there are many closed $1$-forms $\beta$ satisfying $\langle \beta, f\rangle \equiv \omega$, and this non-uniqueness makes it possible to construct examples violating the conclusions of the theorem.
However, given some curious apparent experimental successes with non-forward invariant $U$ (cf. the guineafowl example of \refFig{cutperformancecomp}), it would be interesting to know whether something useful can be said about a subclass of open sets $U$ strictly larger than those considered in the theorem. 
\end{Rem}
\begin{Rem}
When $f\in \Cont^{k+1}$, it is possible to extend Theorem~\ref{th:estimated_isochrons_approach_true_isochrons} to obtain convergence in $\Cont^{k}$ norm rather than merely $\Cont^0$ norm.
The proof is nearly the same, but requires additional care in defining function spaces (since $K$ is an arbitrary compact subset of $\Xs$).
We defer a careful treatment to future work.
\end{Rem}
\begin{proof}[Proof of Theorem~\ref{th:estimated_isochrons_approach_true_isochrons}]
We first establish \eqref{eqn:est-t1f-close-to-true}.
By \refEqn{ulln}, with probability $1$ the function $J_N|_{\F}\colon \F \to [0,\infty)$ converges uniformly to the continuous function $J\colon \F\to [0,\infty)$ defined by $$J(\beta)\coloneqq \int_{K}\left(\langle \beta,f \rangle-\omega\right)^2 d\mu.$$
Hence $$\limsup_{N\to\infty}J(\dTe_N) \stackrel{\textnormal{a.s.}}{=} \limsup_{N\to\infty}J_N(\dTe_N),$$ so Theorem~\ref{th:cost_funcs_close} implies that $$\lim_{\kappa+\epsilon\to 0}\limsup_{N\to\infty}J(\dTe_N) \stackrel{\textnormal{a.s.}}{=} 0.$$
Defining $\mathcal{M}$ to be the set of global minimizers of $J|_{\F}$ (attaining the minimum value of $0$) and using compactness of $\F$, it follows that 
$$\lim_{\kappa+\epsilon\to 0}\limsup_{N\to\infty}\, \dist{\dTe_N}{\mathcal{M}}_{\Cont^0(K)} \stackrel{\textnormal{a.s.}}{=} 0,$$
where $\dist{\alpha}{\mathcal{M}}_{\Cont^0(K)}\coloneqq \inf_{\beta\in \mathcal{M}} \|\alpha-\beta\|_{\Cont^0(K)}$.
(Cf. the Berge maximum theorem \citep[Thm~17.31]{aliprantis2006infinite}.)

Thus, to complete the proof of \eqref{eqn:est-t1f-close-to-true} it suffices to show that the restriction to $U$ of every closed $1$-form in $\mathcal{M}$ coincides with $\dT|_{U}$. 
Since $J$ has minimum $0$ and $U$ is contained in the support of $\mu$, it follows that any $\beta\in \mathcal{M}$ must satisfy $\langle \beta, f\rangle \equiv \omega$.
Since $U$ is forward invariant, the same argument as in the final paragraph of the proof of Theorem~\ref{th:phase-reps-relate} (using Lemmas~\ref{lem:unique} and \ref{lem:t1f-to-circle-phase}) then implies that $\beta = \dT|_U$, as desired.

Next, assume that $\F$ is star-shaped with respect to $\dT$, \eqref{eqn:relint} holds for some $\delta > 0$, and  \eqref{eqn:conv-rates} holds.
Defining $\beta_N\coloneqq \dTe_N - \dT \in -\dT + \F$ for each $N$, Lemma~\ref{lem:coercivity} implies the first inequality below:
\begin{equation}\label{eqn:coercivity-2}
\begin{split}
c \|\beta_N\|_{\Cont^0(U)}^2 &\leq \int_U \langle \beta_N, f\rangle ^2 d\mu\\
&\leq \int_K \langle \beta_N, f\rangle ^2 d\mu = J(\dTe_N).
\end{split}
\end{equation}
The second inequality follows from $U\subset K$, and the equality follows from expanding the integrand defining $J$.
On the other hand, Theorem~\ref{th:cost_funcs_close} implies that
  \begin{align*}
 & \frac{1}{4}\left(\frac{\sqrt{J_N(\dTe_N)}-2\sqrt{\epsilon}}{\|\dT\|_{\Cont^0(K)}+(1+\sqrt{2})\|\dTe_N\|_{\Cont^0(K)}}\right)^2  \leq \frac{1}{N}\sum_{i=1}^N\|\eta_i\|^2.
  \end{align*}
Rearranging and using \eqref{eqn:conv-rates} and \eqref{eqn:coercivity-2} yields the first inequality in \eqref{eqn:form-conv-rate}.
The second inequality in \eqref{eqn:form-conv-rate} follows directly from this and $\|\dT\|_{\Cont^0(K)}, \|\dTe_N\|_{\Cont^0(K)} \leq  M_0$.
\end{proof}


\end{document}